\documentclass[%
  a4paper,
 onecolumn,
  colorlinks,
]{style}

\usepackage[english]{babel}

\usepackage{graphicx}
\usepackage{amsmath}
\usepackage{amssymb}
\usepackage{amsthm}
\usepackage{mathrsfs}
 \usepackage{cleveref}

\def\R{\mathbb{R}}
\def\C{\mathbb{C}}

\def\N{\mathbb{N}}
\def\T{\mathcal{T}}

\def\U{\mathcal{U}}
\def\Y{\mathcal{Y}}

\def\Ri{\mathcal{R}}

\def\tr{\ensuremath{\text{trace}}}
\def\sine{\ensuremath{\text{sin}}}
\def\Re{\ensuremath{\text{Re}}}
\def\sg{\ensuremath{\mathcal{T}}}
\def\gen{\ensuremath{\mathcal{A}}}
\def\ssb{\ensuremath{X}}
\def\sin{\ensuremath{U}}
\def\sout{\ensuremath{Y}}
\def\oin{\ensuremath{\mathcal{B}}}
\def\oout{\ensuremath{\mathcal{C}}}
\def\tf{\ensuremath{\mathcal{G}}}
\def\rtf{\ensuremath{\tf_r}}
\def\mtf{\ensuremath{\mathbf{G}}}
\def\rmtf{\ensuremath{\mathbf{G}_r}}
\def\slm{\ensuremath{W_r}}
\def\srm{\ensuremath{V_r}}
\def\olm{\ensuremath{\mathcal{W}_r}}
\def\orm{\ensuremath{\mathcal{V}_r}}
\def\rEm{\ensuremath{\mathbf{E}_r}}
\def\rAm{\ensuremath{\mathbf{A}_r}}
\def\rBm{\ensuremath{\mathbf{B}_r}}
\def\rCm{\ensuremath{\mathbf{C}_r}}
\def\rEop{\ensuremath{\mathcal{E}_r}}
\def\rAop{\ensuremath{\gen_r}}
\def\rBop{\ensuremath{\oin_r}}
\def\rCop{\ensuremath{\oout_r}}
\def\E{\ensuremath{\mathbf{E}}}
\def\A{\ensuremath{\mathbf{A}}}
\def\B{\ensuremath{\mathbf{B}}}
\def\I{\ensuremath{\mathbf{I}}}
\def\d{\ensuremath{\mathrm{d}}}

\newcommand{\bbm}{\begin{bmatrix}}
\newcommand{\ebm}{\end{bmatrix}}
\newcommand{\bb}[1]{\mathbf{#1}}
\newcommand{\id}{\mathrm{id}}
\newcommand{\dom}{\mathcal{D}}
\newcommand{\ran}{\text{Range}}
\newcommand{\nullsp}{\text{Null}}

\newtheorem{theorem}{Theorem}
\newtheorem{lemma}[theorem]{Lemma}
\newtheorem{corollary}[theorem]{Corollary}
\theoremstyle{definition}
\newtheorem{definition}{Definition}[section]
\newtheorem{remark}[definition]{Remark}

\crefname{theorem}{Theorem}{Theorems}
\Crefname{theorem}{Theorem}{Theorems}
\crefname{corollary}{Corollary}{Corollaries}
\Crefname{corollary}{Corollary}{Corollaries}
\crefname{lemma}{Lemma}{Lemmas}
\Crefname{lemma}{Lemma}{Lemmas}
\crefname{proposition}{Proposition}{Propositions}
\Crefname{proposition}{Proposition}{Propositions}
\crefname{remark}{Remark}{Remarks}
\Crefname{remark}{Remark}{Remarks}
\crefname{definition}{Definition}{Definitions}
\Crefname{definition}{Definition}{Definitions}

\definecolor{cosmicblue}{RGB}{68, 138, 211}
\definecolor{cadmiumgreen}{rgb}{0.0, 0.42, 0.24}

\begin{document}

%%%%%%%%%%%%%%%%%%%%%%%%%%%%%%%%%%%%%%%%%%%%%%%%%%%%%%%%%%%%%%%%%%%%%%%%%%%%%%%%
% PAPER INFORMATION.                                                           %
%%%%%%%%%%%%%%%%%%%%%%%%%%%%%%%%%%%%%%%%%%%%%%%%%%%%%%%%%%%%%%%%%%%%%%%%%%%%%%%%

\title{From Interpolation to $\mathcal{H}_2$ Optimality:\\ Model Reduction for Infinite-Dimensional Linear Control Systems}

\author[$\ast$]{Cankat Tilki}
\affil[$\ast$]{Department of Mathematics, Virginia Tech, Blacksburg, VA, USA.\authorcr
\email{cankat@vt.edu}, \orcid{0009-0001-8667-3075}}

\author[$\dagger$]{Tobias Breiten}
\affil[$\dagger$]{Institut für Mathematik, Technische Universität Berlin, 
Berlin, Germany.\authorcr
\email{tobias.breiten@tu-berlin.de}, \orcid{0000-0002-9815-4897}}

\author[$\ddagger$]{Serkan G{\" u}{\u g}ercin}
\affil[$\ddagger$]{Department of Mathematics, Virginia Tech, Blacksburg, VA, USA.\authorcr
\email{gugercin@vt.edu}, \orcid{0000-0003-4564-5999}}

\shorttitle{Model Reduction for Infinite-Dimensional Systems}
\shortauthor{C. Tilki, T. Breiten, S. G{\" u}{\u g}ercin}
\shortdate{}
\shortinstitute{}



\abstract{
We develop the interpolatory $\mathcal{H}_2$ optimal model reduction framework for linear control systems posed on infinite dimensional state, input and output spaces. Specifically, we consider linear systems formulated as controlled abstract Cauchy problems on a Banach space and approximate them via Petrov-Galerkin projection onto finite dimensional trial and test subspaces. We show that the resulting reduced order transfer function interpolates the original at prescribed points, and we characterize precisely how the projection subspaces must be constructed to enforce this interpolation. Building on this, we develop a data-driven realization framework---an infinite dimensional analogue of the Loewner approach---that recovers the system behavior directly from input-output data without requiring access to the underlying operators. Finally, we derive $\mathcal{H}_2$ optimality conditions for the reduced model and show that the classical interpolatory characterization persists in this infinite dimensional setting: first-order optimality requires Hermite interpolation of the transfer function at the mirror images of the reduced model's poles. Taken together, these results establish that the interpolatory 
$\mathcal{H}_2$ optimal model reduction theory extends naturally and completely to 
infinite dimensional linear control systems with infinite dimensional input and output spaces.
}

\novelty{}

\maketitle

%%%%%%%%%%%%%%%%%%%%%%%%%%%%%%%%%%%%%%%%%%%%%%%%%%%%%%%%%%%%%%%%%%%%%%%%%%%%%%%%
% PAPER CONTENT.                                                               %
%%%%%%%%%%%%%%%%%%%%%%%%%%%%%%%%%%%%%%%%%%%%%%%%%%%%%%%%%%%%%%%%%%%%%%%%%%%%%%%%

\section{Introduction}\label{sec: intro}%

Linear time-invariant (LTI) systems have been studied widely in control theory due to their rich structure, see~\cite{kailath1980linear,ROC_Book,Antoulas_Book} for a detailed overview. Specifically, for fixed matrices $\E,\A \in \R^{n\times n},~ \B\in \R^{n\times m},~\mathbf{C}\in \R^{p\times n} $, an initial condition $x_0\in \R^{n}$ and an input function $u\in L_2(\R_+;\R^m)$ the associated LTI system is defined as
\begin{equation}
\begin{aligned} \label{eq:finiteLTIsys}
    \E\dot{x}(t) &= \A x(t) + \B u(t)\\
    y(t) &= \mathbf{C} x(t)
\end{aligned}~,~~x(0) = x_0,~ t\geq 0,
\end{equation}
where $x(t)\in \R^n$ is the state and $y(t)\in \R^p$ is the output variable. Here by $\mathbb{R}_+$ and $\mathbb{R}_-$ we denote the positive and negative real half-lines, respectively, and by $\mathbb{C}_+$ and $\mathbb{C}_-$ we will denote the open right and left complex half-planes, respectively. Also, for a measure space $M$ and Banach space $K$, $L_p(M;K)$ denotes the space of all Lebesgue measurable, $p$-integrable functions $f\colon M\to K$.
We will assume that $\E$ is invertible. Together with the assumption that $x_0 = 0\in \R^n$, the input-output relation becomes
\begin{equation}\label{eq: fin_dim_io}
    y(t) = \int_{0}^t \mathbf{C} \exp(\widetilde{\A} \tau)\widetilde{\B} u(t-\tau)\d\tau,~\mbox{where}~\widetilde{\A} := \E^{-1}\A~\mbox{and}~\widetilde{\B} := \E^{-1}\B.
\end{equation}
In order for the output trajectory to remain bounded for all $t>0$, we assume that the real parts of the eigenvalues of $\widetilde{\A}$ are negative. 

The input-output relation~\eqref{eq: fin_dim_io} uniquely characterizes the action of the LTI system on any input: the output $y(t)$ is given by the causal convolution of the input function $u(t)$ with the \emph{impulse response} $g(t) := \mathbf{C}e^{\widetilde{\A} t}\widetilde{\B}$.
This time-domain input-output characteristic can be expressed equivalently in the frequency-domain using the \emph{transfer function}. Then for any $f\in L_1(\R_+;\R^p) \cap L_2(\R_+;\R^p)$, the action of the Laplace transform $\mathscr{L}$ is given as
\begin{equation*}
    \mathscr{L}[f](s) = \int_{t=0}^\infty f(t)e^{-st}\d t.
\end{equation*}
Since $L_1(\R_+;\R^p) \cap L_2(\R_+;\R^p)$ is dense in $L_2(\R_+;\R^p)$, one can then extend $\mathscr{L}$ to the entire $L_2(\R_+;\R^p)$ by continuity. Moreover, since the Laplace transform converts causal convolution into multiplication, applying it to \cref{eq: fin_dim_io} yields $\mathscr{L}[y](s) = \mtf(s)\,\mathscr{L}[u](s)$, where
\begin{equation} \label{eq:Gfinite}
    \mtf(s) := \mathbf{C}(s\E-\A)^{-1}\B
\end{equation}
is the \emph{transfer function} of the system. Thus, in the Laplace domain, the input-to-output mapping becomes multiplication by $\mtf$.
 
Note that $\mtf$ is a rational function. For simplicity, assume $\widetilde{\A}$ is diagonalizable with distinct eigenvalues, so $\widetilde{\A} = \mathbf{V}\Lambda \mathbf{V}^{-1}$ and $\lambda_i \neq \lambda_j$ for any $i\neq j$. Let $\mathbf{c}_i$ be the $i^{\text{th}}$ column of $\mathbf{CV}^{-1}$ and $\mathbf{b}_j$ be the $j^{\text{th}}$ column of $(\mathbf{V}\widetilde{\B})^*$. Then we can rewrite \cref{eq:Gfinite} as
\begin{equation}\label{eq: fin_dim_pr}
    \mtf(s) = \mathbf{CV}^{-1}(s\I - \Lambda)^{-1}\mathbf{V}\widetilde{\B} \\
    = \sum_{i=1}^n \dfrac{1}{s-\lambda_i} \mathbf{c}_i\mathbf{b}_i^*.
\end{equation}
Hence, $\mtf$  is a \emph{rational function} with poles on $\{\lambda_i\}_{i=1}^n$ and residues $\mathbf{c}_i\mathbf{b}_i^*\in \C^{p\times m}$. This is called the pole-residue decomposition of $\mtf$ and can be further generalized to cases where $\widetilde{\A}$ does not have distinct eigenvalues, see \cite{kailath1980linear} for details. 

In many cases, LTI systems of this form arise from spatial discretizations of partial differential equations. However, as higher-fidelity discretizations are sought, the state dimension $n$ grows large. Hence the matrices $\E,\A,\B,\mathbf{C}$, can become very high-dimensional. Simulating such systems repeatedly for different inputs then becomes computationally prohibitive. This necessitates the need to find an LTI system of smaller dimension whose input-output behavior closely approximates that of the original system. This is the model reduction problem. To put it rigorously, our goal is to construct a \emph{reduced} system:  for $ r\ll n$, find $\rEm, \rAm \in \R^{r\times r},~ \rBm\in \R^{r\times m},~\rCm\in \R^{p\times r}$ such that the input-output map of
\begin{equation}\label{eq:finiteLTIsysred}
\begin{aligned}
    \rEm\dot{x}_r(t) &= \rAm x_r(t) + \rBm u(t)\\
    y_r(t) &= \rCm x_r(t)
\end{aligned}~,~~x_r(0) = x_{r,0},~t\geq 0,
\end{equation}
is ``close" to the input-output map of the original system \cref{eq:finiteLTIsys}, i.e., $\|y-y_r\|$ is small for some norm of interest and all $u\in L_2(\R_+;\R^m)$.  Here, $x_r(t)\in \R^r$ is the reduced state variable and $y_r(t)\in \R^p$ is the output variable. Assuming 
zero initial condition $x_{r,0}= 0$ as before, as in~\eqref{eq:Gfinite}, the transfer function of the reduced model~\eqref{eq:finiteLTIsysred} 
is given by 
\begin{equation} \label{eq:Grfinite}
\rmtf(s) := \rCm(s\rEm-\rAm)^{-1}\rBm,
\end{equation} 
which is now a degree-$r$ rational function in $s$.  
For LTI systems, the error $\|y-y_r\|$ is directly related to the distance 
between $\rmtf$ and $\mtf$. Indeed, for appropriately chosen norms of interest, one obtains
\begin{equation} \label{eq:yminusyr}
    \|y-y_r\| \leq \|\mtf-\rmtf\|\|u\|.
\end{equation}
We will specify the norm of choice in~\eqref{eq:yminusyr} in later sections. Therefore, for LTI systems finding a high-fidelity, optimal reduced order model boils down to minimizing $\|\mtf - \rmtf\|$, we refer the reader to, e.g., \cite{finH2,ROC_Book,Antoulas_Book}, for details. This turns the model reduction problem into a rational function approximation problem. 

Motivated by this connection, interpolatory model reduction is among the most widely used approaches to the model reduction problem where one constructs the degree-$r$ rational function $\rtf$ that interpolates the original function $\mtf$ at carefully selected points, see \cite{finH2} and the references therein for a comprehensive treatment of such methods. 
Moreover, even when only the transfer function $\mtf$ evaluations are available but without explicit access to the matrices $\E,\A,\B,\mathbf{C}$ for the LTI system \cref{eq:finiteLTIsys}, one can still construct an interpolatory rational approximant $\rmtf$ from the transfer function data using the Loewner framework for interpolation~\cite{LoewnerMayoAntoulas}.
Interpolation is of particular interest because, for the commonly used $\mathcal{H}_2$ notion of optimality to measure the distance between the full and reduced systems,  the transfer function $\rmtf$ of the reduced order model \cref{eq:finiteLTIsysred} is characterized by interpolation of $\mtf$ at the mirror images of its own poles $\lambda_i$~\cite{IRKA,MeierLuenberger,van2008h2,BunKVW10}.  This class of model reduction methods will be our main focus in this paper.

Another class of model reduction methods rooted in systems and control theory uses the concept of system Gramians and Hankel singular values. The two most commonly used methods in this class are balanced truncation (BT) \cite{mullis1976synthesis,moore1981principal} and optimal Hankel-norm approximation (HNA) \cite{HankelNorm}. Both methods come with rigorous a priori error bounds in the so-called $\mathcal{H}_\infty$ norm. Moreover, HNA is optimal in the sense that it minimizes the Hankel norm of the error system $\mtf-\rmtf$. See \cite{Antoulas_Book,breiten2021balancing,benner2017modelch6,ROC_Book} for a comprehensive introduction to BT.

The model reduction problem for LTI systems \cref{eq:finiteLTIsys} is well studied,
with mature theory and practical algorithms spanning the methods mentioned above. A natural next step is to ask how far this theory extends beyond the finite dimensional setting, and indeed several directions have already been explored in the literature.
For balanced truncation, we refer the reader to \cite{GloverCurtain, Bonnet, GuiverOpmeer, ReiS14} where the ``infinite dimensional LTI system'' is formulated as a controlled abstract Cauchy problem. Then they generalized the Gramians, build a balancing transformation for these Gramians and truncate the directions whose contributions are small. A practical implementation of this approach requires solving the associated Lyapunov equations in the infinite dimensional setting, and numerical methods for doing so are investigated in \cite{ADI_Lyap, Grubisic}. Interpolatory methods have also received attention in this direction: moment matching with infinite dimensional state spaces has been investigated in \cite{Moment_Matching,Moment_Matching_Deu}, with results demonstrating close approximation of the original system on practical examples. Structure-preserving reduction has likewise been considered; in \cite{PG_Pass_Stab}, a Petrov-Galerkin framework is used to construct reduced models that provably preserve passivity and stability. For the cases where one can explicitly write down the transfer function of the underlying infinite dimensional problem such as those highlighted in~\cite{CurM09}, one can use the Loewner framework~\cite{LoewnerMayoAntoulas} to construct interpolatory rational approximants directly from transfer function samples. Moreover, for such systems, integrating the Loewner framework 
into $\mathcal{H}_2$ optimal approximation theory,  using only transfer function, including the first-order derivative, samples, \cite{TF-IRKA} constructs rational approximants that satisfy the interpolatory $\mathcal{H}_2$ optimality conditions.
However, in all of these interpolatory model reduction works, the input and output spaces are still assumed to be finite dimensional.

In this paper, we take this next step in full generality. We develop the interpolatory $\mathcal{H}_2$ optimal model reduction theory for LTI systems in which both the state space and the input-output spaces are infinite dimensional, extending the existing framework simultaneously across all three fronts: interpolation theory, data-driven realization, and $\mathcal{H}_2$ optimality conditions. Specifically, we consider systems posed on an arbitrary Banach space with input and output operators acting on separable Hilbert spaces, a canonical situation when formulating nonlinear systems as a Koopman-based linear infinite dimensional system \cite{HovBrei}. We show that the key results underpinning finite dimensional interpolatory $\mathcal{H}_2$ model reduction carry over to this infinite dimensional setting under mild assumptions. Our goal is to obtain a \emph{locally optimal} reduced model with respect to the $\mathcal{H}_2$ norm, and our contributions can be summarized as follows:

For LTI systems posed on a Banach space, with input and output spaces being Hilbert spaces,
\begin{itemize}
   \item in \cref{thm: intPG},  we develop conditions on the trial and test subspaces such that the resulting Petrov-Galerkin projection yields an interpolatory projection, enforcing the transfer function of the reduced system to interpolate that of the original at selected interpolation points;
    \item in \cref{cor: sylv_int}, we reformulate the interpolatory projection conditions of \cref{thm: intPG} as Sylvester-like operator equations, further establishing a structural connection to the classical interpolatory projection framework in the finite dimensional setting;
    \item in \cref{thm: Loewner}, we extend the Loewner framework to the infinite dimensional setting, enabling a data-driven realization of the reduced order system directly from transfer function data; 
    \item in \cref{thm: H2opt}, we establish necessary conditions for $\mathcal{H}_2$ optimal model reduction of infinite dimensional LTI systems, demonstrating that the classical interpolatory optimality conditions at mirror image points carry over naturally to this broader setting.
\end{itemize}
The rest of the paper is organized as follows. In \cref{sect: Prelim}, we introduce notation and mathematical preliminaries. In \cref{sect: IntMOR}, we develop the interpolatory model reduction framework for infinite dimensional LTI systems with a Banach state space and Hilbert input and output spaces. We begin in \cref{ss: cACPdef} by posing the controlled abstract Cauchy problem with output (cACP), then construct the Petrov-Galerkin projection and establish the interpolatory subspace conditions in \cref{ss: intPG}. In \cref{ss: infLoew}, we extend the Loewner framework to the infinite dimensional setting, expressing the reduced order operators directly from transfer function data. We illustrate these results on a numerical example in \cref{ss: infEx}. Finally, in \cref{sect: infH2opt}, we derive the interpolatory necessary conditions for
$\mathcal{H}_2$ optimal approximation in the infinite dimensional setting.

%%%%%%%%%%%%%%%%%%%%%%%%%%%%%%%%%%%%%%%%%%%%%%%%%%%%%%%%%%%%%%%%%%%%%%%%%%%%%%%%
% *** Preliminaries ***                                                        %
%%%%%%%%%%%%%%%%%%%%%%%%%%%%%%%%%%%%%%%%%%%%%%%%%%%%%%%%%%%%%%%%%%%%%%%%%%%%%%%%

\section{Mathematical preliminaries and notation}\label{sect: Prelim}

For a Banach space $(X,\|\cdot\|_X)$, with the norm $\|\cdot\|_X$, let $(X^*,\|\cdot\|_{X^*})$ denote its topological dual space, endowed with the norm induced from $X$, i.e., for any $f\in X^*$ we define
\begin{equation*}
    \|f\|_{X^*} := \sup_{0 \neq x\in X} \dfrac{|f(x)|}{\|x\|_X}.
\end{equation*}
The associated \emph{dual pairing} is defined by evaluation in the natural way and we denote this dual pairing by $\langle f,x\rangle_{(X^*,X)}:= f(x)$, for any $f\in X^*$ and $x\in X$. Note that $\langle \cdot,\cdot \rangle_{(X^*,X)}$ is bilinear.
The annihilator of $V\subseteq X$, denoted $V^\perp$, is defined as
\begin{equation*}
    V^\perp := \left\{ f\in X^*\mid \langle f,x\rangle_{(X^*,X)} = 0 \text{ for all }x\in V\right\}\subseteq X^*.
\end{equation*}
Similarly, the pre-annihilator of $W\subseteq X^*$, denoted ${}^\perp W$, is defined as
\begin{equation*}
    {}^\perp W := \left\{ x\in X\mid \langle f,x\rangle_{(X^*,X)} = 0 \text{ for all }f\in W\right\}\subseteq X.
\end{equation*}

One important property of Banach spaces is their embedding into the double dual. This embedding $J\colon X \mapsto (X^*)^*$ is given as follows. Given an element $x\in X$, let $J_x\colon X^*\to \mathbb{C}$ be defined for any $f\in X^*$ as $J_x(f):=f(x)\in \mathbb{C}$. It then follows that $J_x$ is linear and bounded with respect to the $\|\cdot\|_{X^*}$ norm, hence $J_x \in (X^*)^*$. The embedding then is defined for any $x\in X$ as $J[x] = J_x\in (X^*)^*$.

For a Hilbert space $U$ with the norm $\|\cdot \|_U$, we will denote the inner product by $\langle\cdot,\cdot\rangle_U$ and adopt the convention where the inner product is linear in the first argument and conjugate-linear in the second when the Hilbert space is complex. Note that since Hilbert spaces are also Banach spaces, the dual space $U^*$ is well defined and should be understood as a Banach space dual.

For a Hilbert space $U$, the Riesz representation theorem \cite[Section III.6]{Yosida} provides a canonical identification of the space with its dual. This identification is through the Riesz map $\Ri_U\colon U\to U^*$ that is defined for any $x\in U$ as
\begin{equation}\label{eq: riesz_map}
    \Ri_U[x](\cdot) = \langle\cdot,x\rangle_{U}\in U^*.
\end{equation}
One can show that $\Ri_U\colon U\to U^*$ is an isometry between $U$ and $U^*$, hence $U^*$ can be identified with $U$ isometrically. In this paper, by an abuse of notation we will not distinguish between those elements when it is clear from the context. However, one should note that this Riesz mapping is a \emph{conjugate-linear mapping} for complex Hilbert spaces \cite[Corollary III.6.1]{Yosida} and a linear mapping for real Hilbert spaces. The Riesz representation theorem also gives us the following relation between the dual pairing and the inner product for any $f\in U^*$ and $x\in U$:
\begin{equation*}
    \langle f,x\rangle_{(U^*,U)} = f(x) = \langle x,f\rangle_{U}.
\end{equation*}
Given two Banach spaces $X,Y$, we will denote the space of bounded linear operators as $\mathcal{L}(X,Y)$. We will reserve $\id_X\in \mathcal{L}(X,X)$ for the identity operator on the Banach space $X$, i.e., $\id_X(x) = x$ for all $x\in X$.
For a linear operator $\mathcal{A}\colon \dom(\mathcal{A})\subseteq X\to Y$ between Banach spaces $X,Y$ we denote its domain as $\mathcal{D}(\mathcal{A})\subseteq X$. Moreover, when $\mathcal{A}$ is densely defined, i.e., $\overline{\dom{(\mathcal{A})}} = X$, we denote its (Banach) adjoint as $\mathcal{A}^*\colon \mathcal{D}(\mathcal{A}^*) \subseteq Y^*\to X^*$. For a densely defined, linear operator $\mathcal{A}\colon \dom{(\mathcal{A})}\subseteq U\to Z$ between Hilbert spaces $U,Z$, we denote its Hilbert adjoint by $\mathcal{A}^\dagger\colon \mathcal{D}(\mathcal{A}^\dagger) \subseteq Z\to U$. One should note that when $X,Y$ are also Hilbert spaces, one can identify the Banach adjoint with the Hilbert adjoint via the Riesz map as follows
\begin{equation}\label{eq: riesz_adj}
    \mathcal{A}^\dagger = \Ri^{-1}_X\mathcal{A}^*\Ri_Y.
\end{equation}
For a linear operator $\mathcal{A}$ on Banach space $X$, we denote its fundamental subspaces as
\begin{align*}
    \text{Range}(\mathcal{A}) = \{\mathcal{A}x\mid x\in \mathcal{D}(\mathcal{A})\}\quad\mbox{and}
    \quad\nullsp(\mathcal{A}) = \{x\in X\mid \mathcal{A}x=0\}.
\end{align*}
Let $\mathcal{A}\in \mathcal{L}(X,Y)$. Then $\nullsp(\mathcal{A})$ will be a closed subspace. Moreover, we have the following identity~\cite[Lemma 3.1.16]{BanachRef}
\begin{equation}\label{eq: null_range}
    \nullsp(\mathcal{A}) = {}^\perp\ran(\mathcal{A}^*). 
\end{equation}
A bounded \emph{projector} is a bounded linear operator $\mathcal{P}\colon X\to X$ satisfying $\mathcal{P}^2 = \mathcal{P}$. It is well known that $\id_X - \mathcal{P}$ is then also a projector, and every $f\in X$ admits the decomposition $f = \mathcal{P}[f] + (\id_X - \mathcal{P})[f]$.

Given a linear operator $\mathcal{A}\colon\mathcal{D}(\mathcal{A})\subseteq X\to X$ and $\lambda\in \mathbb{C}$ we define $\mathcal{A}_\lambda := \lambda~\id_X - \mathcal{A}$. When $\mathcal{A}_\lambda$ is injective, it is invertible over its range and the resolvent $(\lambda\,\id_X - \mathcal{A})^{-1}$ is well defined. The \emph{resolvent set} of $\mathcal{A}$, denoted $\rho(\mathcal{A})$, consists of all $\lambda\in\mathbb{C}$ for which $(\lambda\,\id_X - \mathcal{A})^{-1}$ extends to a bounded operator on all of $X$, i.e., $\overline{\mathcal{D}((\lambda\,\id_X - \mathcal{A})^{-1})} = X$.
One can then define the \emph{spectrum} of $\mathcal{A}$ as $\sigma(\mathcal{A}) = \mathbb{C}\backslash\rho(\mathcal{A})$.

A \emph{$C_0$-semigroup} on a Banach space $X$ is a family of bounded linear operators $(\mathcal{T}_t)_{t\geq 0}$ on $X$ satisfying $\mathcal{T}_0 = \mathrm{id}_X$, $\mathcal{T}_t\mathcal{T}_s = \mathcal{T}_{t+s}$, and $\|\mathcal{T}_t x_0 - x_0\|_X\to 0$ as $t\to 0^+$ for all $x_0\in X$. Its generator $\gen\colon \mathcal{D}(\mathcal{A}) \subset X \to X$ is defined as
\begin{equation*}
    \mathcal{A}[x] := \lim_{t\downarrow 0} \frac{\mathcal{T}_t [x] - x}{t},
\end{equation*}
where the limit is taken in the strong operator topology, and the domain $\mathcal{D}(\mathcal{A})$ consists of all $x\in X$ for which this limit exists.
Note that $\gen$ does not have to be bounded but is still a linear operator. Actually, one can show that any $C_0$-semigroup generator $\mathcal{A}$ is closed and densely defined~\cite[Theorem II.3.8]{EngelNagel}. With these $C_0$-semigroups, for any initial condition $x_0\in X$, we can then pose the \emph{abstract Cauchy problem (ACP)}:
\begin{equation}\label{eq: ACP}
    \dot{x}(t) = \gen x(t),~x(0) = x_0,\quad x_0\in \dom(\gen),~t \geq 0.
\end{equation}
Since $(\sg_t)_{t\geq 0}$ is a $C_0$-semigroup, \cref{eq: ACP} is well-posed~\cite[Corollary II.6.9]{EngelNagel}. In addition, we will assume that the semigroup $(\sg_t)_{t\geq 0}$ is exponentially stable, i.e., for some $\beta <0$ and $M> 0$ we have
\begin{equation*}
    \|\sg_t\|\leq Me^{\beta t},~\text{for }t \geq 0.
\end{equation*}
For a Banach space $X$ and a measurable, exponentially bounded function $f\colon \R_+\to X$ with exponent $w\in \R$, i.e., $\|f(t)\|_X\leq Ce^{wt}$ for all $t\geq 0 $ and some constant $C>0$ we define the \emph{Laplace transform} $\mathscr{L}\colon \{s\in \C \mid \text{Re}(s) > w\}\to \sout$ as \cite[Definition C.15]{EngelNagel}
\begin{equation*}
    \mathscr{L}[f](s) := \int_{t=0}^\infty e^{-st}f(t)\,\mathrm{d}t. 
\end{equation*}

\section{Interpolatory model reduction via Petrov-Galerkin projection on Banach spaces}\label{sect: IntMOR}
In this section, we analyze the input-output system formulated as an abstract Cauchy problem, where the state evolution is governed by a $C_0$-semigroup, and the output is defined through an additional observation operator. We will carry out this analysis by extending the frequency domain framework of finite dimensional linear time-invariant systems to infinite dimensions. Then, we will discuss how to develop Petrov-Galerkin projections for this setting. Lastly, we will prove with specific trial and test subspaces, Petrov-Galerkin projection amounts to interpolating the transfer function - a well known result for finite dimensional linear time-invariant systems.

\subsection{Model reduction problem for infinite dimensional LTI systems}\label{ss: cACPdef}

Given an exponentially stable 
$C_0$-semigroup $\{\sg_t\}_{t\geq 0}$ with exponent $\beta < 0$ and $M> 0$ on a Banach space $\ssb$, let $\gen$ be its generator with domain $\dom(\gen)\subset \ssb$. Moreover, let $\sin,\sout$ be input and output Hilbert spaces respectively and let $\oin\colon \sin\to \ssb$, $\oout\colon\ssb\to \sout$ be bounded linear operators. Let $u\in L_1(\R_+;\sin)\cap L_2(\R_+;\sin)$. We can then formulate the following inhomogeneous, controlled abstract Cauchy problem coupled with an output equation
\begin{equation}\label{eq: cACP}
    \begin{aligned}
        \dot{x}(t) &= \gen x(t) + \oin u(t) \\
        y(t) &= \oout x(t)\\
    \end{aligned}~,~~x(0) = x_0\in \dom(\gen),~t\geq 0.
\end{equation}
Here, $x(t)\in \ssb$ is the state variable and $y(t)\in \sout$ is the output variable. This can be considered as an infinite dimensional analog of linear time-invariant systems as given in~\eqref{eq:finiteLTIsys}. For \cref{eq: cACP}, the mild solution \cite[Definition VI.7.2]{EngelNagel} $\Phi\colon \R_+\times \ssb \to \ssb$ is defined for any initial condition $x_0\in \ssb$ and time $t\in \R_+$ as
\begin{equation*}
    \Phi(t,x_0) = \T_t[x_0] + \int_{\tau = 0}^t \sg_{t-\tau} \oin[u(\tau)]\,\mathrm{d}\tau.
\end{equation*}
Then, for $x_0 = 0$ the \emph{input-output map} becomes
\begin{equation}\label{eq: iomap}
    u\mapsto    y(t) = \oout[x(t)] = \int_{\tau = 0}^t \oout\sg_{t-\tau} \oin[u(\tau)]\,\mathrm{d}\tau.
\end{equation}
Since, $u\in L_1(\R_+;\sin)\cap L_2(\R_+;\sin)$ and the $C_0$-semigroup $(\sg_t)_{t\geq 0}$ is exponentially stable, the Laplace transform of $y$ is well defined for $s\in \C_+$. Moreover, by using the properties \cite[Theorem C.17 \& Theorem II.1.10(i)]{EngelNagel} we can express the Laplace transform of the input-output map as
\begin{equation*}
    \mathscr{L}[y](s) =\underbrace{\oout (s~\id_\ssb - \mathcal{A})^{-1} \oin}_{\tf(s)}\left[\mathscr{L}\left[ u\right](s)\right].
\end{equation*}
This operator between the Laplace transform of the input and output function is the \emph{transfer function} $\tf\colon\C\backslash\sigma(\gen)\to \mathcal{L}(\sin,\sout)$ of this system and is given as
\begin{equation}\label{eq: TF}
    \tf(s) = \oout(s~\id_\ssb- \gen)^{-1}\oin.
\end{equation}
Note that $\tf$ is a holomorphic function on $\C\backslash\sigma(\gen)$~\cite{LaplBanach} and therefore infinitely Fr{\' e}chet differentiable on $\C\backslash\sigma(\gen)$. Moreover, its Fr{\' e}chet derivatives are given by 
\begin{equation}\label{eq: derTF}
    \dfrac{\mathrm{d}^n}{\mathrm{d}s^n}\tf(s) = (-1)^n n!\, \oout(s~\id_\ssb-\gen)^{-(n+1)}\oin.
\end{equation}
Using these concepts, we can pose the model reduction problem for infinite dimensional LTI systems as calculating a \emph{reduced} cACP for \cref{eq: cACP}, i.e., for a (small) finite $r\in \mathbb N$, find $\rEop,\rAop \in \C^{r\times r},~ \rBop\colon\sin \to \C^r,~\rCop\colon \C^r \to \sout$ such that the input-output map of
\begin{equation}\label{eq: red_cACP}
\begin{aligned}
    \rEop\dot{x}_r(t) &= \rAop x_r(t) + \rBop u(t)\\
    y_r(t) &= \rCop x_r(t)
\end{aligned}~,~~x_r(0) = 0\in \C^r,~t\geq 0
\end{equation}
is `close' to the input-output map of the original system \cref{eq:finiteLTIsys}, i.e., $\|y-y_r\|$ is small for some norm of interest and all $u\in L_2(\R_+;\sin)$. Here, $x_r(t)\in \C^r$ is the reduced state variable and $y_r(t)\in \sout$ is the output variable as before. 

\subsection{Petrov-Galerkin projection and interpolatory model reduction}\label{ss: intPG}
For finite dimensional LTI systems \cref{eq:finiteLTIsys}, the most common way to obtain a reduced order model \cref{eq:finiteLTIsysred} is via the Petrov-Galerkin projection framework: one fixes two subspaces $V_r,W_r\subseteq \C^n$ and constructs matrices $\mathbf{V}_r,\mathbf{W}_r\in \C^{n\times r}$ such that $\ran{(\mathbf{V}_r)} = V_r$ and $\ran{(\mathbf{W}_r)} = W_r$  to obtain the reduced matrices $\rEm, \rAm,\rBm,\rCm$ from the full-order model \cref{eq:finiteLTIsys} via projection as
\begin{equation}  \label{eq:finiteproj}
    \rEm = \mathbf{W}_r^* \E \mathbf{V}_r,~\rAm = \mathbf{W}_r^*\A\mathbf{V}_r,~\rBm = \mathbf{W}_r^*\B,~\rCm = \mathbf{C}\mathbf{V}_r.
\end{equation}
Moreover, interpolation of the transfer function $\mtf$ of \cref{eq:finiteLTIsys} by the reduced transfer function $\rmtf$ of \cref{eq:finiteLTIsysred} can be enforced by an appropriate choice of the model reduction bases $\mathbf{V}_r$ and $\mathbf{W}_r$; see \cite[Theorem 3.3.1]{finH2}. Building on this idea, in this section we extend this interpolation framework to the setting where the state space is a Banach space and the input and output spaces are Hilbert spaces.

Fix $r$-dimensional subspaces $\srm \subset \ssb$ and $\slm\subset \ssb^*$. We will refer to these subspaces as the trial and test subspaces, respectively. Let $\srm$ and $\slm$ be spanned by linearly independent non-zero vectors $\{v_1,\dots,v_r\}\subset \dom(\gen)\subseteq \ssb$ and $\{w_1,\dots,w_r\}\subset \dom(\gen^*)\subseteq \ssb^*$ and without loss of generality assume that $\|v_i\|_X = \|w_i\|_{X^*} = 1$ for any $i = 1,\dots,r$. Then we define the linear operators $\orm\colon\C^r\to \ssb$ and $\olm\colon \ssb\to \C^r$ as
\begin{equation}\label{eq: VrWr}
    \orm\alpha = \sum_{i=1}^r v_i\alpha_i,\quad \mathcal{W}_r[x] = \bbm \langle w_1,x\rangle_{(\ssb^*,\ssb)} \\  \vdots \\ \langle w_r,x\rangle_{(\ssb^*,\ssb)}\ebm,\quad \alpha\in \C^r,~x\in \ssb.
\end{equation}
Using these operators, we can define a reduced system \cref{eq: red_cACP} with
the corresponding reduced operators $\rEop,\rAop \in \C^{r\times r},~\rBop\colon\sin\to\C^r$ and $\rCop \colon\C^r \to \sout$ defined as
\begin{equation}\label{eq: red_mat}
    \rEop = \olm\orm,~\rAop = \olm\gen\orm,~ \rBop= \olm\oin,~ \rCop = \oout\orm.
\end{equation}
Note that one can easily verify that $\mathcal{B}_r\in \mathcal{L}(U,\mathbb C^r),\mathcal{C}_r\in \mathcal{L}(\mathbb C^r,Y).$
Coordinate-wise, \eqref{eq: red_mat} can be rewritten as, for $i,j = 1,\dots,r$,
\begin{align} 
    \rEop^{(i,j)} &= \langle w_i,v_j\rangle_{(\ssb^*,\ssb)},~\rAop^{(i,j)} = \langle w_i,\gen[v_j]\rangle_{(\ssb^*,\ssb)},
    \label{eq: red_mat_ErAr}
    \\ \rBop^i[f] &= \langle w_i,\oin[f]\rangle_{(\ssb^*,\ssb)},~~\mbox{and}~~ \rCop^j = \oout[v_j],
    \label{eq: red_mat_BrCr}
\end{align}
where $\rEop^{(i,j)}$ and $\rAop^{(i,j)}$ denote the $(i,j)$th entry of $\rEop$ and $\rAop$, $\rBop^i[f]$ denotes the $i$th column of $\rBop[f]$, and $\rCop^j := \rCop[e_j]$ with $e_j\in \C^r$ being the $j$th canonical basis vector in $\C^r$. 

Assume that $\text{det}(\rEop) = \text{det}(\olm\orm) \neq 0$. This assumption ensures that $$\srm\cap \slm^{\perp} = \{0\}\subset \ssb^{**},$$ where elements of $\ssb$ are identified canonically with their images in $\ssb^{**}$. Indeed, for contradiction assume that $\srm\cap \slm^{\perp} \neq \{0\}$. Hence one can pick $z\neq 0\in \ssb$ such that $z \in \srm\cap \slm^{\perp}$. Since $z\in \srm$ one can find an $\alpha \neq 0 \in \C^r$ such that $\orm[\alpha] = z$. Also since $J_z\in \slm^\perp$ we have \begin{equation*}
     0 = \langle w_i, z\rangle_{(\ssb^{*},\ssb)} = w_i(x) = J_z(w_i) = \langle J_z, w_i\rangle_{(\ssb^{**},\ssb^*)}
\end{equation*} 
for any $i = 1,\dots ,r$. So we get $\olm[z] = 0$. Thus we have $\olm\orm[\alpha] =0$. Since $\text{det}(\olm\orm) \neq 0 $ this implies $\alpha = 0$, contradicting the assumption.

For the reduced system, the transfer function $\rtf\colon\C\backslash\sigma(\rEop^{-1}\rAop)\to\mathcal{L}(\sin,\sout)$ becomes
\begin{equation}\label{eq: redTF}
    \rtf(s) = \rCop(s\rEop - \rAop)^{-1}\rBop,
\end{equation}
where $(s\rEop - \rAop)^{-1}$ is now a degree-$r$ rational function in $s$. Before stating the main result of this section, we establish the following lemma, which will be used in the proof of \Cref{thm: intPG}.

\begin{lemma}\label{thm: intPGlem}
    For fixed $s$ in the resolvent sets of $\rEop^{-1}\rAop$ and $\gen$, define $\mathcal{P}_r(s)\colon\dom(\gen)\subset \ssb \to \ssb$ and $\mathcal{Q}_r(s)\colon\ssb\to \ssb$ as
    \begin{align}
            \mathcal{P}_r(s) &= \orm(s\rEop-\rAop)^{-1}\olm(s~\id_\ssb-\gen), \label{eq: Pr}\\ 
            \mathcal{Q}_r(s) &= (s~\id_\ssb-\gen)\mathcal{P}_r(z)(s~\id_\ssb-\gen)^{-1} = (s~\id_\ssb-\gen)\orm(s\rEop-\rAop)^{-1}\olm,\label{eq: Qr}
    \end{align}
    where $\olm,\orm$ are defined as in \cref{eq: VrWr}, $\gen$  as in \cref{eq: cACP}, and $\rEop,\rAop$ as in \cref{eq: red_mat}. Then $\mathcal{Q}_r(s)$ is a bounded projector on $\ssb$ and $\mathcal{P}_r(s)$ can be extended uniquely to a bounded projector on $\ssb$. Moreover, $\ran(\mathcal{P}_r(s)) = \srm$ and $\ran(\mathcal{Q}_r(s)^*) = \slm$. 
\end{lemma}
\begin{proof}
    First, note that $\mathcal{Q}_r(s)$ is well defined since the range of $\mathcal{P}_r(s)$ is contained in $\srm \subset \dom(\gen)$. For arbitrary $\alpha\in \mathbb C^r$, consider
    \begin{equation*}
        \|(s~\id_{\ssb}-\gen)\orm[\alpha]\|_\ssb =  \left\|\sum_{i=1}^r(s~\id_{\ssb}-\gen)v_i \alpha_i\right\|_\ssb \le \left(\sum_{i=1}^r \| (s~\id_{\ssb}-\gen)v_i \|_\ssb\right) \| \alpha\|,
    \end{equation*}
    which shows that $(s~\id_{\ssb}-\gen)\orm \in \mathcal{B}(\C^r,\ssb).$ Write $\|(s\rEop-\rAop)^{-1}\olm\|$ as
    \begin{equation*}
        \|(s\rEop-\rAop)^{-1}\olm\| = \sup_{x\in\ssb} \dfrac{\|(s\rEop-\rAop)^{-1}\olm[x]\|_{\C^r}}{\|x\|_{\ssb}}.
    \end{equation*}
    Then, since $\|\olm[x]\|_{\C^r} = \sqrt{\sum_{i=1}^r|\langle w_i,x\rangle_{(\ssb^*,\ssb)}|^2} \leq r \underset{i=1,\dots r}{\max} |\langle w_i,x\rangle_{(\ssb^*,\ssb)}| \leq r \|x\|_{\ssb^*}$, we have
    \begin{equation}\label{eq: ineq_resnorm}
        \sup_{x\in\ssb} \dfrac{\|(s\rEop-\rAop)^{-1}\olm[x]\|_{\C^r}}{\|x\|_{\ssb}} \leq \sup_{x\in\ssb} \dfrac{r\|(s\rEop-\rAop)^{-1}\olm[x]\|_{\C^r}}{\|\olm[x]\|_{\C^r}}.
    \end{equation}
    If we let $\alpha_x := \olm[x] \in \C^r$ we can bound the right-hand side of \cref{eq: ineq_resnorm} as
    \begin{equation*}
    \begin{aligned}
        \sup_{x\in\ssb} \dfrac{r\|(s\rEop-\rAop)^{-1}\olm[x]\|_{\C^r}}{\|\olm[x]\|_{\C^r}} &= \sup_{\alpha_x\in~\ran(\olm)} \dfrac{r\|(s\rEop-\rAop)^{-1}\alpha_x\|_{\C^r}}{\|\alpha_x\|_{\C^r}} \\
        & \leq \sup_{\alpha\in \C^r} \dfrac{r\|(s\rEop-\rAop)^{-1}\alpha\|_{\C^r}}{\|\alpha\|_{\C^r}} \\
        & = r\|(s\rEop-\rAop)^{-1}\| <\infty.
    \end{aligned}
    \end{equation*}
    Since $s\notin \sigma(\rAop)$, $(s\rEop-\rAop)^{-1}\olm \in \mathcal{B}(\ssb,\C^r)$. Then, their composition
    \begin{equation*}
        (s~\id_{\ssb}-\gen)\orm(s\rEop-\rAop)^{-1}\olm = \mathcal{Q}_r(s)
    \end{equation*}
    is also bounded, showing that $\mathcal{Q}_r(s)\in \mathcal{L}(\ssb,\ssb)$.
    
    To show that $\mathcal{P}_r(s)$ is bounded on $\dom(\gen) \subset \ssb$ and has a bounded extension $\widetilde{\mathcal{P}}_r(s)\in \mathcal{B}(\ssb,\ssb)$ we will split $\mathcal{P}_r(s) = \orm(s\rEop-\rAop)^{-1}\olm(s~\id_\ssb-\gen)$ in two parts and show each part is bounded. For $\olm(s~\id_{\ssb}-\gen)$, given any $x\in \dom(\gen)$ we get
    \begin{equation}\label{eq: WrresA}
        \dfrac{\|\olm(s~\id_{\ssb}-\gen)[x]\|_{\C^r}}{\|x\|_\ssb} = \dfrac{\|s\olm[x] - \olm\gen[x]\|_{\C^r}}{\|x\|_{\ssb}}\leq \dfrac{\|s\olm[x]\|_{\C^r}}{\|x\|_{\ssb}} + \dfrac{\|\olm\gen[x]\|_{\C^r}}{\|x\|_{\ssb}}.
    \end{equation}
    We first consider the first term in the upper bound~\eqref{eq: WrresA} and rewrite it as 
    \begin{equation}\label{eq: WrBdd}
        \dfrac{\|s\olm[x]\|_{\C^r}}{\|x\|_{\ssb}} = |s|\dfrac{\|\olm[x]\|_{\C^r}}{\|x\|_{\ssb}} = |s|\dfrac{\left\|\bbm\langle w_1,x \rangle & \cdots & \langle w_n,x \rangle\ebm^\top\right\|_{\C^r}}{\|x\|_{\ssb}}.
    \end{equation}
    Since for the dual norm we have $|\langle w_i,x \rangle| \leq \|w_i\|_{\ssb^*}\|x\|_{\ssb}$, we find
    \begin{equation*}
    \begin{aligned}
        |s|\dfrac{\left\|\bbm\langle w_1,x \rangle & \cdots & \langle w_n,x \rangle\ebm^\top\right\|_{\C^r}}{\|x\|_{\ssb}} &\leq |s|\dfrac{\left\|\bbm \|w_1\|_{\ssb^*}\|x\|_{\ssb} & \cdots & \|w_n\|_{\ssb^*}\|x\|_{\ssb}\ebm^\top\right\|_{\C^r}}{\|x\|_{\ssb}} \\
        &= |s|{\left\|\bbm \|w_1\|_{\ssb^*} & \cdots & \|w_n\|_{\ssb^*}\ebm^\top\right\|_{\C^r}}\dfrac{\|x\|_{\ssb}}{\|x\|_{\ssb}} \\
        &= |s|\sqrt{\sum_{i=1}^r\|w_i\|^2_{\ssb^*}} < \infty.
    \end{aligned}
    \end{equation*}
    Now we consider the second term in the upper bound $\dfrac{\|\olm\gen[x]\|_{\C^r}}{\|x\|_{\ssb}}$. Since $\olm$ is defined as \cref{eq: VrWr}, the numerator can be expressed as
    \begin{equation*}
        \|\olm\gen[x]\|_{\C^r} = \sqrt{\sum_{i=1}^r|\langle w_i,\gen[x]\rangle|^2}.
    \end{equation*}
    By the definition of the adjoint and since $w_i\in \dom(\gen^*)$, we have
    \begin{equation*}
        \dfrac{\|\olm\gen[x]\|_{\C^r}}{\|x\|_{\ssb}} = \dfrac{\sqrt{\sum_{i=1}^r|\langle w_i,\gen[x]\rangle|^2}}{\|x\|_{\ssb}} = \dfrac{\sqrt{\sum_{i=1}^r|\langle \gen^*[w_i],x\rangle|^2}}{\|x\|_\ssb}.
    \end{equation*}
    Then, using the dual norm definition, we get
    \begin{equation*}
        \dfrac{\|\olm\gen[x]\|_{\C^r}}{\|x\|_{\ssb}} = \dfrac{\sqrt{\sum_{i=1}^r|\langle \gen^*[w_i],x\rangle|^2}}{\|x\|_\ssb} \leq \sqrt{\sum_{i=1}^r\|\gen^*[w_i]\|_{\ssb^*}^2}.
    \end{equation*}
    Since both components of the upper bound~\cref{eq: WrresA} are bounded, $\olm(s~\id_X -\gen)\in \mathcal{L}(\dom(\gen),\C^r)$ is also bounded on $\dom(\gen)$. Then, since $\olm(s~\id_{\ssb}-\gen)\in \mathcal{L}(\dom(\gen),\C^r)$ and $\orm(s\rEop-\rAop)^{-1}\in \mathcal{L}(\C^r,\ssb)$, we obtain that $\mathcal{P}_r(s):\dom(\gen)\subseteq X \to \ssb$ is a bounded operator on $\dom(\gen)$.
    
    Since $\dom(\mathcal{P}_r(s)) = \dom(\gen)$, it is only a bounded operator on $\dom(\gen)\subset \ssb$. However, $\gen$ is densely defined, i.e., $\overline{\dom(\gen)} = \ssb$. Hence, $\mathcal{P}_r(s)$ has a unique extension to a bounded operator $\widetilde{\mathcal{P}_r}(s)$ on $\ssb$.

    We can also show that $\widetilde{\mathcal{P}_r}(s)$ and $\mathcal{Q}_r(s)$ are projectors for any $s\in \rho(\mathcal{A})\cap \rho(\rEop^{-1}\rAop)$. Indeed we have
    \begin{equation*}
        \begin{aligned}
        \mathcal{P}_r^2(s) &= \orm(s\rEop-\rAop)^{-1}\underbrace{\olm(s~\id_{\ssb}-\gen)\orm}_{s\rEop - \rAop}(s\rEop-\rAop)^{-1}\olm(s~\id_{\ssb}-\gen) \\
        &= \orm(s\rEop-\rAop)^{-1}(s\rEop - \rAop)(s\rEop-\rAop)^{-1}\olm(s~\id_{\ssb}-\gen) \\
        &= \orm(s\rEop-\rAop)^{-1}\olm(s~\id_{\ssb}-\gen) = \mathcal{P}_r(s).
        \end{aligned}
    \end{equation*}
    Similarly for $\mathcal{Q}_r(s)$ we have
    \begin{equation*}
        \begin{aligned}
        \mathcal{Q}_r^2(s) &= (s~\id_\ssb-\gen)\orm(s\rEop-\rAop)^{-1}\underbrace{\olm(s~\id_\ssb-\gen)\orm}_{s\rEop - \rAop}(s\rEop-\rAop)^{-1}\olm \\
        &= (s~\id_\ssb-\gen)\orm(s\rEop-\rAop)^{-1}\olm = \mathcal{Q}_r(s).
        \end{aligned}
    \end{equation*}
    Then, since $\dom(\gen)$ is dense in $X$, we have that $\widetilde{\mathcal{P}_r}(s)$ and $\mathcal{Q}_r(s)$ are bounded projectors. For the ranges, note that since $s\in\rho(\gen)\cap \rho(\rEop^{-1}\rAop)$, both $s~\id_{\ssb} - \gen$ and $s\rEop - \rAop$ are invertible. Moreover since $w_i$ are linearly independent, $\olm$ has rank $r$. Hence $\srm =\ran(\mathcal{P}_r(s))$, so $\ran(\mathcal{P}_r(s))$ is finite dimensional and hence closed, i.e.,
    \begin{equation*}
        \overline{\ran(\mathcal{P}_r(s))} = \ran(\mathcal{P}_r(s)).
    \end{equation*}
    Since $\widetilde{\mathcal{P}_r}(s)$ is the continuous extension of $\mathcal{P}_r(s)$ to $X$, we then have 
    \begin{equation*}
        \ran(\widetilde{\mathcal{P}_r}(s)) = \overline{\ran(\mathcal{P}_r(s))} = \ran(\mathcal{P}_r(s)) = \srm.
    \end{equation*}
    For $\ran(\mathcal{Q}_r(s)^*)$, consider \cref{eq: null_range}. For $\mathcal{Q}_r(s)$ it reads as
    \begin{equation*}
        \nullsp(\mathcal{Q}_r(s)) = {}^\perp\ran(\mathcal{Q}_r(s)^*).
    \end{equation*}
    Taking annihilators of both sides, we get
    \begin{equation*}
        \nullsp(\mathcal{Q}_r(s))^\perp = ({}^\perp\ran(\mathcal{Q}_r(s)^*))^\perp.
    \end{equation*}
    Then by \cite[Proposition 2.6.6(c)]{BanachRef}, this becomes
    \begin{equation}\label{eq: w*Qr}
         \overline{\ran(\mathcal{Q}_r(s)^*)}^{w^*} = \nullsp(\mathcal{Q}_r(s))^\perp,
    \end{equation}
    where $\overline{\ran(\mathcal{Q}_r(s)^*)}^{w^*}$ denotes the weak$^*$ closure of $\ran(\mathcal{Q}_r(s)^*)$. By the same argument applied to $\olm$, we can also show that
    \begin{equation}\label{eq: w*Wr}
         \overline{\ran(\olm^*)}^{w^*} = \nullsp(\olm)^\perp.
    \end{equation}
    Note that $\olm^*\colon\C^r\to X^*$ can be expressed for any $\alpha\in \C^r$ as    
    \begin{equation}\label{eq: Wr_adj}
        \olm^*[\alpha] = \sum_{i=1}^r\alpha_iw_i.
    \end{equation}
    Indeed, for any $\alpha\in \C^r,x\in X$ we have
    \begin{equation*}
        \left\langle\sum_{i=1}^r\alpha_iw_i, x\right\rangle_{(X^*,X)} = \sum_{i=1}^r\alpha_i\left\langle w_i, x\right\rangle_{(X^*,X)}.
    \end{equation*}
    By \cref{eq: VrWr}, we can rewrite this as
    \begin{equation*}
        \sum_{i=1}^r\alpha_i\left\langle w_i, x\right\rangle_{(X^*,X)} = \alpha^\top\olm[x].
    \end{equation*}
    Since the dot product is the dual pairing for $\Bigl((\C^r)^*,\C^r\Bigr)$, we get
    \begin{equation*}
        \alpha^\top\olm[x] = \langle\alpha, \olm[x]\rangle_{((\C^r)^*,\C^r)}.
    \end{equation*}
    $\olm^*$, given in \cref{eq: Wr_adj}, is a rank-$r$ operator with range $\slm$. Since $\slm$ is a finite dimensional subspace of $X^*$, it is weak$^*$ closed. Hence we have
    \begin{equation*}
        \ran(\olm^*) = \olm = \overline{\ran(\olm^*)}^{w^*}.
    \end{equation*}
    Combining this with \cref{eq: w*Wr} we get 
    \begin{equation}\label{eq: null_perp_Wr}
        \slm = \overline{\ran(\olm^*)}^{w^*} = \nullsp(\olm)^\perp.
    \end{equation}
    Consider \cref{eq: Qr}. Since $s\in\rho(\gen)\cap \rho(\rEop^{-1}\rAop)$, $s~\id_{\ssb} - \gen$ and $s\rEop - \rAop$ are invertible. Moreover $\orm$ is assumed to be rank-$r$. Hence,
    \begin{equation*}
        \nullsp(\mathcal{Q}_r(s)) = \nullsp(\olm).
    \end{equation*}
    Taking the annihilators of both sides we get
    \begin{equation*}
        \nullsp(\mathcal{Q}_r(s))^\perp = \nullsp(\olm)^\perp.
    \end{equation*}
    Combining this with \cref{eq: null_perp_Wr} we get
    \begin{equation*}
        \nullsp(\mathcal{Q}_r(s))^\perp = \nullsp(\olm)^\perp =\slm.
    \end{equation*}
    Then by \cref{eq: w*Qr} this yields
    \begin{equation}\label{eq: w*Qr_Wr}
        \overline{\ran(\mathcal{Q}_r(s)^*)}^{w^*} = \nullsp(\mathcal{Q}_r(s))^\perp = \slm.
    \end{equation}
    Since $\slm$ is a finite dimensional subspace, it is weak$^*$ closed. Hence \cref{eq: w*Qr_Wr} becomes
    \begin{equation*}
        \ran(\mathcal{Q}_r(s)^*) = \slm.
    \end{equation*}
\end{proof}

We are now in a position to state \Cref{thm: intPG}, which connects the trial and test subspaces $\srm$ and $\slm$ to interpolation of the transfer function 
$\tf$.

\begin{theorem}[Interpolation by Petrov-Galerkin projection]\label{thm: intPG}
    Let $\gen,\oin,\oout$ be defined by the cACP \cref{eq: cACP} and the reduced ones $\rEop,\rAop,\rBop,\rCop$ be defined by \cref{eq: red_mat}. Define their transfer functions by $\tf$ and $\rtf$ respectively. Consider $\varrho,\mu \in \C$ such that both $(s~\id_X - \gen)$ and $(s\rEop-\rAop)$ are invertible for $s \in \{\varrho,\mu\}.$ Moreover, let $\olm,\orm$ have full rank and assume $p\in \sin$ and $q\in \sout$ are both nonzero. Then,
    \begin{enumerate}
        \item if $(\varrho~\id_\ssb-\gen)^{-1}\oin[p] \in \srm$, then $\tf(\varrho)[p] = \rtf(\varrho)[p]$,
        \item if $(\mu~\id_\ssb-\gen)^{-*}\oout^*[q] \in \slm$, then $\tf(\mu)^*[q] = \rtf(\mu)^*[q]$,
        \item if both assumptions 1 and 2 hold, and $\varrho = \mu$, then 
        \begin{equation*}
            \left\langle \dfrac{\mathrm{d}}{\mathrm{d}s}\tf(\varrho)[p], q\right\rangle_{\sout} = \left\langle \dfrac{\mathrm{d}}{\mathrm{d}s}\rtf(\varrho)[p], q\right\rangle_{\sout}.
        \end{equation*}
    \end{enumerate}
\end{theorem}
\begin{proof}
    Fix $s$ in the resolvent sets of $\rEop^{-1}\rAop$ and $\gen$ and consider 
    the extension of $\mathcal{P}_r(s)$ on $\ssb$. Here, we will abuse the notation and denote the unique extension of $\mathcal{P}_r(s)$ on $\ssb$ as $\mathcal{P}_r(s)$. By \Cref{thm: intPGlem}, $\mathcal{P}_r(s)$ and $\mathcal{Q}_r(s)$ are bounded projectors, $\ran(\mathcal{P}_r(s)) = \srm$ and $\ran(\mathcal{Q}_r(s)^*) = \slm$. Moreover, in a neighborhood of $\varrho$ (or $\mu$), both $\mathcal{P}_r(s)$ and $\mathcal{Q}_r(s)$ are holomorphic, operator-valued functions. 
    
    Consider the error transfer function $\tf- \rtf$. By \cref{eq: TF} and \cref{eq: redTF} we can write it as
    \begin{equation*}
        \tf(s)- \rtf(s) = \oout(s~\id_\ssb-\gen)^{-1}\oin - \rCop(s\rEop-\rAop)^{-1}\rBop.
    \end{equation*}
    Using \cref{eq: red_mat} for the reduced operators, we obtain
    \begin{equation*}
        \oout(s~\id_\ssb-\gen)^{-1}\oin - \rCop(s\rEop-\rAop)^{-1}\rBop = \oout((s~\id_\ssb-\gen)^{-1} - \orm(s\rEop-\rAop)^{-1}\olm)\oin .
    \end{equation*}
    Since $s\notin \sigma(\gen)$ we can apply the resolvent of $\gen$ and its inverse in $s$ to obtain
    \begin{equation*}
    \begin{aligned}
        \oout((s~\id_\ssb-\gen)^{-1}& - \orm(s\rEop-\rAop)^{-1}\olm)\oin = \\ &\oout((s~\id_\ssb-\gen)^{-1} - (s~\id_\ssb-\gen)^{-1}(s~\id_\ssb-\gen)\orm(s\rEop-\rAop)^{-1}\olm)\oin.
    \end{aligned}
    \end{equation*}
    By the definition of $\mathcal{Q}_r(s)$ in \cref{eq: Qr} we can write this as
    \begin{equation*}
    \begin{aligned}
        \oout((s~\id_\ssb-\gen)^{-1} - (s~\id_\ssb-\gen)^{-1}(s~\id_\ssb-\gen)\orm&(s\rEop-\rAop)^{-1}\olm)\oin = \\&\oout(s~\id_\ssb-\gen)^{-1}(\id_\ssb-\mathcal{Q}_r(s))\oin.
    \end{aligned}
    \end{equation*}
    Since $(\id_\ssb-\mathcal{Q}_r(s))$ is a projector, we have
    \begin{equation*}
        \oout(s~\id_\ssb-\gen)^{-1}(\id_\ssb-\mathcal{Q}_r(s))\oin = \oout(s~\id_\ssb-\gen)^{-1}(\id_\ssb-\mathcal{Q}_r(s))^2\oin.
    \end{equation*}
    Then, by \cref{eq: Qr} we can write the last term as
    \begin{equation*}
    \begin{aligned}
        \oout(s~\id_\ssb-\gen)^{-1}&(\id_\ssb-\mathcal{Q}_r(s))^2\oin = \\& \oout(s~\id_\ssb-\gen)^{-1}(\id_\ssb-\mathcal{Q}_r(s))(s~\id_\ssb-\gen)(\id_\ssb - \mathcal{P}_r(s))(s~\id_\ssb-\gen)^{-1}\oin.
        \end{aligned}
    \end{equation*}
    So, in total we get
    \begin{equation}\label{eq: err_tf}
        \tf(s)- \rtf(s) = \oout(s~\id_\ssb-\gen)^{-1}(\id_\ssb-\mathcal{Q}_r(s))(s~\id_\ssb-\gen)(\id_\ssb - \mathcal{P}_r(s))(s~\id_\ssb-\gen)^{-1}\oin.
    \end{equation}
    Evaluating \cref{eq: err_tf} at $\varrho\in \C$ and $p\in \sin$ yields
    \begin{equation*}
    \begin{aligned}
        (\tf(\varrho)- &\rtf(\varrho))[p] = \\
        &= \oout(\varrho~\id_\ssb-\gen)^{-1}(\id_\ssb-\mathcal{Q}_r(\varrho))(\varrho~\id_\ssb-\gen)(\id_\ssb - \mathcal{P}_r(\varrho))\underbrace{(\varrho~\id_\ssb-\gen)^{-1}\oin[p]}_{\in \srm}.
    \end{aligned}
    \end{equation*}
    By assumption $(\varrho~\id_\ssb-\gen)^{-1}\oin[p]\in \srm$. Since $\mathcal{P}_r(\varrho)$ is a projection on $\srm$, $(\id_\ssb - \mathcal{P}_r(\varrho))(\varrho~\id_\ssb-\gen)^{-1}\oin[p] = 0$. Hence, the above expression evaluates to zero, proving the first assertion.

    For the second assertion, considering the adjoint of \eqref{eq: err_tf} yields
    \begin{equation*}
        (\tf(s)- \rtf(s))^* = \oin^*(s~\id_\ssb-\gen)^{-*}(\id_\ssb-\mathcal{P}_r(z))^*(s~\id_\ssb-\gen)^*(\id_\ssb-\mathcal{Q}_r(s))^*(s~\id_\ssb-\gen)^{-*}\oout^*.
    \end{equation*}
    Evaluating this expression at $\mu\in \C$ and $q\in \sout$ we get
    \begin{equation*}
    \begin{aligned}
        (\tf(\mu)&- \rtf(\mu))^* =\\
        &= \oin^*(\mu~\id_\ssb-\gen)^{-*}(\id_\ssb-\mathcal{P}_r(\mu))^*(\mu~\id_\ssb-\gen)^*(\id_\ssb-\mathcal{Q}_r(\mu))^*\underbrace{(\mu~\id_\ssb-\gen)^{-*}\oout^*[q]}_{\in \slm}.
    \end{aligned}
    \end{equation*}
    Again, by assumption $(\mu~\id_\ssb-\gen)^{-*}\oout^*[q] \in \slm$. Since $$\slm = \text{Range}(\mathcal{Q}_r(\mu)^*) = \text{Null}((\id_\ssb-\mathcal{Q}_r(\mu))^*)$$ we have $(\id_X-\mathcal{Q}_r(\mu))^*(\mu~\id_\ssb-\gen)^{-*}\oout^*[q] = 0$. Hence, the above expression evaluates to zero, proving the second assertion.

To prove the third (and the final) assertion, we first follow \cite[Section 3.11]{HillePhilips} and write the Taylor expansions 
    \begin{align}
        ((\varrho+\varepsilon \gamma)~\id_\ssb-\gen)^{-1} &= (\varrho~\id_\ssb -\gen)^{-1} + \varepsilon \gamma(\varrho~\id_\ssb-\gen)^{-2} + \mathcal{O}(\varepsilon^2),\label{eq: TaylA} \\
        ((\varrho+\varepsilon \gamma)\rEop-\rAop)^{-1} &= (\varrho\rEop -\rAop)^{-1} + \varepsilon \gamma(\varrho\rEop-\rAop)^{-1}\rEop(\varrho \rEop-\rAop)^{-1} + \mathcal{O}(\varepsilon^2),\label{eq: TaylAr}
    \end{align}
    where $\gamma\in \C$ is assumed to have unit norm, and 
    $\varepsilon\in \R_+$ is small enough.
    Then by evaluating \eqref{eq: err_tf} at $\varrho_\varepsilon := \varrho+\varepsilon \gamma$ we have
    \begin{align*}
    \tf(\varrho_\varepsilon)-&   \rtf(\varrho_\varepsilon ) = \\
        &\oout(\varrho_\varepsilon~\id_\ssb-\gen)^{-1}(\id_\ssb-\mathcal{Q}_r(\varrho_\varepsilon))(\varrho_\varepsilon~\id_\ssb-\gen)(\id_\ssb-\mathcal{P}_r(\varrho_\varepsilon))(\varrho_\varepsilon ~\id_\ssb-\gen)^{-1}\oin.
    \end{align*}
    Then using \cref{eq: TaylA} for both inverses, \cref{eq: err_tf} and splitting $(\varrho_\varepsilon~\id_\ssb-\gen)$ as
    \begin{equation*}
        (\varrho_\varepsilon~\id_\ssb-\gen) = (\varrho~\id_\ssb-\gen) + \varepsilon\gamma~\id_\ssb
    \end{equation*}
    we get
    \begin{equation}\label{eq: tfminustfr}
        \tf(\varrho_\varepsilon)- \rtf(\varrho_\varepsilon )= (\tf-\rtf)(\varrho) + \varepsilon\gamma N + \varepsilon \gamma M_1 + \varepsilon \gamma M_2 + \mathcal{O}(\varepsilon^2),
    \end{equation}
    where $N$, $M_1$ and $M_2$ are defined as
    \begin{align}
        N   &:= \oout(\varrho~\id_\ssb-\gen)^{-1}(\id_\ssb-\mathcal{Q}_r(\varrho_\varepsilon))(\id_\ssb-\mathcal{P}_r(\varrho_\varepsilon))(\varrho~\id_\ssb-\gen)^{-1}\oin, \label{eq: N}\\
        M_1 &:=  \oout(\varrho~\id_\ssb-\gen)^{-2}(\id_\ssb-\mathcal{Q}_r(\varrho_\varepsilon))(\varrho_\varepsilon~\id_\ssb-\gen)(\id_\ssb-\mathcal{P}_r(\varrho_\varepsilon))(\varrho~\id_\ssb-\gen)^{-1}\oin, \label{eq: M1}\\
        M_2 &:= \oout(\varrho~\id_\ssb-\gen)^{-1}(\id_\ssb-\mathcal{Q}_r(\varrho_\varepsilon))(\varrho_\varepsilon~\id_\ssb-\gen)(\id_\ssb-\mathcal{P}_r(\varrho_\varepsilon))(\varrho~\id_\ssb-\gen)^{-2}\oin.\label{eq: M2}
    \end{align}
    Evaluating \cref{eq: tfminustfr} at $p$ and then taking its inner product with $q$ results in
    \begin{equation*}
    \begin{aligned}
        \langle (\tf(\varrho_\varepsilon) - \rtf(\varrho_\varepsilon))[p],q\rangle_{\sout} = \langle (\tf-\rtf)(\varrho)[p] + \varepsilon \gamma N[p] + \varepsilon \gamma M_1[p] + \varepsilon \gamma M_2[p]+ \mathcal{O}(\varepsilon^2), q \rangle_{\sout}  \\
         = \langle (\tf-\rtf)(\varrho)[p], q\rangle_{\sout} + \varepsilon \gamma \langle N[p], q\rangle_{\sout} + \varepsilon \gamma \langle M_1[p], q\rangle_{\sout} + \varepsilon \gamma \langle  M_2[p], q \rangle_{\sout} + \mathcal{O}(\varepsilon^2). 
    \end{aligned}
    \end{equation*}
    We will calculate this last equation term by term. For the first term, note that $(\tf-\rtf)(\varrho)[p] = 0$ as proven before. Hence the first term will evaluate to zero. For the second term we write $N[p]$ via \cref{eq: N}
    \begin{equation*}
        N[p] = \oout(\varrho~\id_\ssb-\gen)^{-1}(\id_\ssb-\mathcal{Q}_r(\varrho_\varepsilon))(\id_\ssb-\mathcal{P}_r(\varrho_\varepsilon))\underbrace{(\varrho~\id_\ssb-\gen)^{-1}\oin[p]}_{\in \srm}.
    \end{equation*}
    Since $(\varrho~\id_\ssb-\gen)^{-1}\oin[p]\in \srm$ and $\mathcal{P}_r$ is a projection on $\srm$, we have that $N[p] = 0$.
    For the third term, we expand $M_1[p]$ via \cref{eq: M1},
    \begin{equation*}
        M_1[p] = \oout(\varrho~\id_\ssb-\gen)^{-2}(\id_\ssb-\mathcal{Q}_r(\varrho_\varepsilon))(\varrho_\varepsilon~\id_\ssb-\gen)(\id_\ssb-\mathcal{P}_r(\varrho_\varepsilon))\underbrace{(\varrho~\id_\ssb-\gen)^{-1}\oin[p]}_{\in \srm}.
    \end{equation*}
    Again, since $(\varrho~\id_\ssb-\gen)^{-1}\oin[p]\in \srm$ and $\mathcal{P}_r$ is a projection on $\srm$. Hence, the third term is also zero. For the last term, we rewrite the inner product as
    \begin{equation*}
        \langle M_2[p], q \rangle_{\sout} = \langle p, M_2^\dagger[q]\rangle_{\sin}.
    \end{equation*}
    Then by expanding $M_2$ in \cref{eq: M2} and taking the adjoint, we get
    \begin{equation*}
    \begin{aligned}
        M_2^\dagger[q] = \Ri_{\sin}^{-1}\oin^*((\varrho~\id_\ssb-\gen)^{-*})^2&(\id_\ssb-\mathcal{P}_r(\varrho_\varepsilon))^*(\varrho_\varepsilon~\id_\ssb-\gen)^*\times
        \\&\times(\id_\ssb-\mathcal{Q}_r(\varrho_\varepsilon))^*\underbrace{(\varrho~\id_\ssb-\gen)^{-*}\oout^*[q]}_{\in \slm} = 0.
    \end{aligned}
    \end{equation*}
    Here $\Ri_{\sin}$ denotes the Riesz map for $\sin$ defined in \cref{eq: riesz_map}. Hence, the third term also evaluates to zero. Summarizing, we get
    \begin{equation*}
        \langle (\tf(\varrho_\varepsilon) - \rtf(\varrho_\varepsilon))[p], q\rangle_{\sout} = \mathcal{O}(\varepsilon^2).
    \end{equation*}
    Dividing both sides by $\varepsilon$ and letting $\varepsilon \to 0$ we obtain
    \begin{equation*}
        \left\langle \dfrac{\mathrm{d}}{\mathrm{d}s}\tf(\varrho)[p], q\right\rangle_{\sout} = \left\langle \dfrac{\mathrm{d}}{\mathrm{d}s}\rtf(\varrho)[p], q\right\rangle_{\sout},
    \end{equation*}
    thus proving the final assertion.
\end{proof}
\begin{remark}
    Note that by the identification between Banach space adjoint and Hilbert space adjoint \cref{eq: riesz_adj}, the second assertion in \cref{thm: intPG} is equivalent to its Hilbert space counterpart, namely if $(\mu~\id_\ssb-\gen)^{-\dagger}\oout^\dagger[q] \in \slm$, then $\tf(\mu)^\dagger[q] = \rtf(\mu)^\dagger[q]$.
\end{remark}

\begin{remark}[\cref{thm: intPG} in $\C^n$]
    The interpolatory projections for model reduction of LTI systems \cref{eq:finiteLTIsys} can be interpreted as a special case of \cref{thm: intPG}. Indeed, if we consider $\C^n$ as a Hilbert space with the inner product $\langle\bb{x},\bb{z}\rangle_{\C^n} = \bb{z}^*\bb{x}$,
    and fix our basis as the standard basis $\{e_j\}_{j=1}^n$, the operators can be represented by the state space matrices $\E,\A,\B,\bb{C}$ of~\cref{eq:finiteLTIsys}, the directions $p,q$ become (the tangential) vectors $\bb{p},\bb{q}$, and \cref{thm: intPG} reduces to the well known interpolation conditions for finite dimensional LTI~\cite[Theorem 3.3.1]{finH2}. More specifically, given the finite dimensional LTI system~\eqref{eq:finiteLTIsys} with the transfer function
    $\mtf$, let the reduced LTI system~\eqref{eq: fin_dim_pr} with the transfer function $\rmtf$ be obtained via projection as in~
    \eqref{eq:finiteproj}. Then, we obtain: 
    \begin{enumerate}
        \item if $(\varrho\E-\A)^{-1}\B \bb{p} \in \ran{(\mathbf{V}_r)}$, then ${\mtf}(\varrho)\bb{p} = {\rmtf}(\varrho)\bb{p}$;
        \item if $(\mu\E-\A)^{-*}\mathbf{C}^* \bb{q} \in \ran{(\mathbf{W}_r)}$, then $\bb{q}^*\bb{\mtf}(\mu) = \bb{q}^*{\rmtf}(\mu)$;
        \item if both assumptions 1 and 2 hold, and $\varrho = \mu$ then 
        \begin{equation*}
            \bb{q}^*\left(\dfrac{\mathrm{d}}{\mathrm{d}s}\bb{\mtf}(\varrho)\right)\bb{p} = \bb{q}^*\left( \dfrac{\mathrm{d}}{\mathrm{d}s}{\rmtf}(\varrho)\right)\bb{p}.
        \end{equation*}
    \end{enumerate}
\end{remark}

\cref{thm: intPG} covers only a single interpolation point and direction. For a multi point interpolation, we have the following corollary.
\begin{corollary}\label{cor: mpInt}
    Given $r \in \N$, for every $ i =1,\dots,r$, pick $\varrho_i,\sigma_i\in \C\backslash\sigma(\gen)$, $p_i\in \sin$ and $q_i\in \sout$. Let $\orm\colon\C^r\to \ssb$ and $\olm\colon\ssb\to \C^r$ be defined as
    \begin{equation}\label{eq: fixVrWr}
    \begin{aligned}
        \orm[\alpha] &:= \sum_{i=1}^r \alpha_i~(\sigma_{i}~\id_\ssb - \gen)^{-1}\oin[p_i],\\
        \olm[x] &:= \bbm\langle(\varrho_{1}~\id_\ssb - \gen)^{-*}\oout^*[q_{1}],x\rangle \\\langle(\varrho_{2}~\id_\ssb - \gen)^{-*}\oout^*[q_{2}],x\rangle \\
        \vdots \\
        \langle(\varrho_{r}~\id_\ssb - \gen)^{-*}\oout^*[q_{r}],x\rangle\ebm.
    \end{aligned}
    \end{equation}
    Thus, pick $v_i = (\sigma_{i}~\id_\ssb - \gen)^{-1}\oin[p_i]$ and $w_i = (\varrho_i~\id_\ssb - \gen)^{-*}\oout^*[q_i]$. Then the reduced cACP \cref{eq: red_cACP} with these operators $\orm, \olm$ satisfy
    \begin{equation*}
        \tf(\sigma_i)[p_i] = \rtf(\sigma_i)[p_i]\quad \mbox{and}\quad
        \tf(\varrho_i)^*[q_i] = \rtf(\varrho_i)^*[q_i]
    \end{equation*}
    for $i=1,\dots, r$. Moreover if $\sigma_i = \varrho_j$ for some $j=1,\dots,r$ we have
    \begin{equation*}
        \left\langle \dfrac{\mathrm{d}}{\mathrm{d}s}\tf(\sigma_i)[p_i], q_j \right\rangle_{\sout} = \left\langle \dfrac{\mathrm{d}}{\mathrm{d}s}\rtf(\sigma_i)[p_i],q_j\right\rangle_{\sout}.
    \end{equation*}
\end{corollary}
Note that in \cref{cor: mpInt}, we choose a specific formulation for $\olm,\orm$. However, it is easy to connect this to $\widetilde{\olm},\widetilde{\orm}$ that satisfy the conditions of \cref{thm: intPG}. It is easy to verify 
    $\widetilde{\olm} = \mathbf{M}\olm$ and $\widetilde{\orm} = \orm\mathbf{K}$
with some invertible matrices $\mathbf{M},\mathbf{K}\in \C^{r\times r}$.
The reduced operators $\widetilde{\rEop},\widetilde{\rAop},\widetilde{\rBop},\widetilde{\rCop}$ resulting from $\widetilde{\olm},\widetilde{\orm}$ can then be written in terms of $\rEop,\rAop,\rBop,\rCop$ as
\begin{equation*}
    \widetilde{\rEop} = \mathbf{M}\rEop\mathbf{K},~\widetilde{\rAop} = \mathbf{M}\rAop\mathbf{K},~\widetilde{\rBop} = \mathbf{M}\rBop,~\widetilde{\rCop} = \rCop\mathbf{K},
\end{equation*}
yielding the same transfer function, i.e.,
$$ \widetilde{\rtf}(s) = \widetilde{\rCop}(s\widetilde{\rEop} - \widetilde{\rAop})^{-1}\widetilde{\rBop} 
 = \rCop(s\rEop - \rAop)^{-1}\rBop = \rtf(s),
$$
by factoring out $\mathbf{M}$ and $\mathbf{K}$.

In finite dimensional LTI systems, one can equivalently describe 
$\olm,\orm$ of \cref{cor: mpInt} via a Sylvester equation~\cite{gallivan2004sylvester},\cite[Theorem 3.3.1]{finH2}.  For the infinite dimensional case we consider here, we can similarly rephrase \Cref{cor: mpInt} as a Sylvester-like equation as shown next.
\begin{corollary}[Sylvester equations for interpolation]\label{cor: sylv_int}
    Fix $\varrho_i,\mu_i\in \C\backslash\sigma(\gen)$, $p_i\in \sin$ and $q_i\in \sout$ for $i = 1,2,\dots,r$, and define $\Sigma_{\varrho},\Sigma_\mu \in \C^{r\times r}$ as $\Sigma_{\varrho}:= \mathrm{diag}(\varrho_1,\dots,\varrho_r)$ and $\Sigma_\mu = \mathrm{diag}(\mu_1,\dots,\mu_r)$. Let $\olm$ and $\orm$ be as in \cref{eq: fixVrWr}. Then $\orm$ and $\olm$ solve the Sylvester-like operator equations
    \begin{equation*}
    \begin{aligned}
        \orm\Sigma_\rho - \gen\orm &= \bbm \oin[p_1] & \cdots & \oin[p_r]\ebm, \\
        \Sigma_\mu\olm - \olm\gen &= \bbm \oout^*[q_1] \\ \vdots \\ \oout^*[q_r] \ebm.
    \end{aligned}
    \end{equation*}
\end{corollary}
\begin{proof} 
    Let $e_i$ be the $i$th canonical basis vector in $\C^r$.  Then we have
    \begin{equation*}
        \begin{aligned}
            \orm\Sigma_re_i - \gen\mathcal{V}_r[e_i] &= \varrho_i v_i - \gen[v_i] 
            = (\varrho~\id_\ssb- \gen)[v_i] \\
            &= (\varrho~\id_\ssb- \gen)(\varrho~\id_\ssb-\gen)^{-1}\oin[p_i] \\
            &= \oin[p_i] 
            = \bbm \oin[p_1] & \cdots & \oin[p_r]\ebm e_i.
        \end{aligned}
    \end{equation*}
    Applying this for all $i = 1,\dots,r$ we get the result. For $\olm$, we apply $e_i$ from the left.
    \begin{equation*}
        \begin{aligned}
            e_i^\top\Sigma_r\olm - e_i^\top\olm\gen &= \mu_i\langle w_i,\cdot\rangle - \langle w_i,\gen[\cdot]\rangle = (\mu_i~\id_\ssb- \gen)^*[w_i] \\
            &= (\mu_i~\id_\ssb- \gen)^*(\mu_i~\id_\ssb-\gen)^{-*}\oout^*[q_i] \\ &= \oout^*[q_i] = e_i^\top \bbm \oout^*[q_1] \\ \vdots \\ \oout^*[q_r]\ebm.
        \end{aligned}
    \end{equation*}
\end{proof}

\subsection{Obtaining the reduced system operators via transfer function data}\label{ss: infLoew}

As in the finite dimensional case, the interpolatory projection framework of \cref{thm: intPG} requires explicit access access to the full-order operators. This might not be always possible; instead dynamics might be only accessible via the (input/output) samples of the transfer function $\tf$. In the finite dimensional case, the Loewner framework provides the solution, constructing the interpolatory reduced system directly from data; see, \cite[Chapter 4]{finH2},\cite{LoewnerMayoAntoulas}.
Inspired by the finite dimensional case, 
we will next introduce a solution to obtain the reduced operators $\rEop,\rAop,\rBop,\rCop$ resulting from \cref{cor: mpInt} solely from the transfer function evaluations $\{\tf(\sigma_i)[p_i]\}_{i=1}^r\cup\{\tf(\varrho_j)^\dagger[q_j]\}_{j=1}^r$. 
\begin{theorem}[Loewner Framework for cACP problem \cref{eq: cACP}]\label{thm: Loewner}
    Given $r \in \N$, for every $ i,j =1,\dots,r$, pick $\sigma_j,\varrho_i\in \C\backslash\sigma(\gen)$, $p_j\in \sin$ and $q_i\in \sout$ and let $\orm\colon\C^r\to \ssb$ and $\olm\colon\ssb\to \C^r$ be constructed as in \cref{eq: fixVrWr} in \cref{cor: mpInt}. Then for $\sigma_j\neq \varrho_i $ the reduced operators $\rEop,\rAop,\rBop,\rCop$ for the reduced abstract Cauchy problem \cref{eq: red_cACP} obtained via explicit Petrov-Galerkin projection using  $\orm,\olm$ can be written  as
    \begin{align}
        \rEop^{(i,j)} &= -\dfrac{\langle p_j,\tf(\varrho_i)^\dagger[q_i]\rangle_{\sin}-\langle \tf(\sigma_j)[p_j], q_i\rangle_{\sout}}{\varrho_i - \sigma_j}, \label{eq:Erdata}\\ 
        \rAop^{(i,j)} &= -\dfrac{\varrho_i\langle p_j,\tf(\varrho_i)^\dagger[q_i]\rangle_{\sin}-\sigma_j\langle \tf(\sigma_j)[p_j], q_i\rangle_{\sout}}{\varrho_i - \sigma_j}, \label{eq:Ardata}
        \\
        \rBop^{(i)}[\cdot]&=\langle \cdot,\tf(\varrho_i)^\dagger[q_i]\rangle_{\sout},  \quad \mbox{and}\quad \rCop^{(j)} =  \tf(\sigma_j)[p_j], \label{eq:BrCrdata}
    \end{align}
    for $i,j = 1,\dots,r$. In addition, for the case where $\sigma_i = \varrho_j$ for some $i,j = 1,\dots,r$, the construction \cref{eq:Erdata} and \cref{eq:Ardata} are replaced by
    \begin{align}
        \rEop^{(i,j)} &= -\left\langle \dfrac{\d}{\d s}\tf(\sigma_j)[p_j], q_i\right\rangle_{\sout} \label{eq:Errepateddata},\\ 
        \rAop^{(i,j)} &= -\left\langle \tf(\sigma_j)[p_j] + \sigma_j\dfrac{\d}{\d s}\tf(\sigma_j)[p_j], q_i\right\rangle_{\sout}.
        \label{eq:Arrepateddata}
    \end{align}
\end{theorem}
\begin{proof}
    Note that $\rEop^{(i,j)} = e_i^\top\rEop e_j = e_i^\top\olm\orm[e_j]$. From \cref{eq: fixVrWr} we have \begin{align*}
    e_i^\top\olm &= \langle w_i,\cdot\rangle = \langle (\varrho_{i}~\id_\ssb - \gen)^{-*}\oout^*[q_{i}],\cdot\rangle \\ \orm[e_j] &= v_j = (\sigma_{j}~\id_\ssb - \gen)^{-1}\oin[p_j].
    \end{align*}
    Hence $\rEop^{(i,j)}$ becomes
    \begin{equation*}
    \begin{aligned}
        \rEop^{(i,j)} &= e_i^\top\olm\orm[e_j]
        = \langle w_i,v_j\rangle_{(\ssb^*, \ssb)}\\
        &= \langle (\varrho_{i}~\id_\ssb - \gen)^{-*}\oout^*[q_{i}],(\sigma_{j}~\id_\ssb - \gen)^{-1}\oin[p_j]\rangle_{(\ssb^*, \ssb)}.
    \end{aligned}
    \end{equation*}
    By the definition of adjoint, we can rewrite this as
    \begin{equation*}
    \begin{aligned}
        \langle (\varrho_{i}~\id_\ssb - \gen)^{-*}\oout^*[q_{i}],&(\sigma_{j}~\id_\ssb - \gen)^{-1}\oin[p_j]\rangle_{(\ssb^*, \ssb)} \\
        &= \langle\oin^*(\sigma_j~\id_\ssb - \gen)^{-*}(\varrho_{i}~\id_\ssb - \gen)^{-*}\oout^*[q_i],p_j\rangle_{(\sin^*,\sin)}. 
    \end{aligned}
    \end{equation*}
    Then using the distribution property of adjoint, we obtain
    \begin{equation*}
    \begin{aligned}
        \langle\oin^*(\sigma_j~\id_\ssb - \gen)^{-*}&(\varrho_{i}~\id_\ssb - \gen)^{-*}\oout^*[q_i],p_j\rangle_{(\sin^*,\sin)} \\&= \langle\left(\oout(\varrho_{i}~\id_\ssb - \gen)^{-1}(\sigma_j~\id_\ssb - \gen)^{-1}\oin\right)^*[q_i],p_j\rangle_{(\sin^*,\sin)}.
    \end{aligned}
    \end{equation*}
    By the Riesz representation theorem, this dual product can be expressed as an inner product on $\sin$:
    \begin{equation*}
    \begin{aligned}
        \langle\left(\oout(\varrho_{i}~\id_\ssb - \gen)^{-1}(\sigma_j~\id_\ssb - \gen)^{-1}\oin\right)^*&[q_i],p_j\rangle_{(\sin^*,\sin)} 
                    \\ 
                    &= \langle p_j, \left(\oout(\varrho_{i}~\id_\ssb - \gen)^{-1}(\sigma_j~\id_\ssb - \gen)^{-1}\oin\right)^\dagger[q_i]\rangle_{\sin} \\
                    &= \langle \oout(\varrho_{i}~\id_\ssb - \gen)^{-1}(\sigma_j~\id_\ssb - \gen)^{-1}\oin[p_j] , q_i\rangle_{\sout}.
    \end{aligned}
    \end{equation*}
    So, in summary, we have
    \begin{equation*}
        \rEop^{(i,j)} = \langle \oout(\varrho_{i}~\id_\ssb - \gen)^{-1}(\sigma_j~\id_\ssb - \gen)^{-1}\oin[p_j] , q_i\rangle_{\sout}.
    \end{equation*}
    For $\sigma_j=\varrho_i$, we use \cref{eq: derTF} to rewrite this as
    \begin{equation*}
        \rEop^{(i,j)} = \langle \oout(\sigma_j~\id_\ssb - \gen)^{-1}(\sigma_j~\id_\ssb - \gen)^{-1}\oin[p_j] , q_i\rangle_{\sout} = -\left\langle \dfrac{\d}{\d s}\tf(\sigma_j)[p_j] , q_i\right\rangle_{\sout}.
    \end{equation*}
    proving \cref{eq:Errepateddata}. For $\sigma_j\neq \varrho_i$ we use the resolvent identity \cite[Theorem 4.8.1]{HillePhilips} to get
    \begin{equation*}
    \begin{aligned}
        \rEop^{(i,j)} &= \langle \oout(\varrho_{i}~\id_\ssb - \gen)^{-1}(\sigma_j~\id_\ssb - \gen)^{-1}\oin[p_j] , q_i\rangle_{\sout} \\
                    &= -\left\langle \oout\left(\dfrac{(\varrho_{i}~\id_\ssb - \gen)^{-1} - (\sigma_j~\id_\ssb - \gen)^{-1}}{\varrho_i - \sigma_j}\right)\oin[p_j] , q_i\right\rangle_{\sout} \\
                    &= -\left\langle \dfrac{\oout(\varrho_{i}~\id_\ssb - \gen)^{-1}\oin[p_j] - \oout(\sigma_j~\id_\ssb - \gen)^{-1}\oin[p_j]}{\varrho_i - \sigma_j} , q_i\right\rangle_{\sout}  \\
                    &= -\dfrac{\langle \oout(\varrho_{i}~\id_\ssb - \gen)^{-1}\oin[p_j] , q_i\rangle_{\sout} - \langle \oout(\sigma_j~\id_\ssb - \gen)^{-1}\oin[p_j], q_i \rangle_{\sout}}{\varrho_i - \sigma_j} \\ 
                    &= -\dfrac{\langle \tf(\varrho_i)[p_j] , q_i\rangle_{\sout} - \langle \tf(\sigma_j)[p_j], q_i \rangle_{\sout}}{\varrho_i - \sigma_j}.
    \end{aligned}
    \end{equation*}
    \begin{equation*}
        \rEop^{(i,j)} = -\dfrac{\langle p_j , \tf(\varrho_i)^\dagger[q_i]\rangle_{\sin} - \langle \tf(\sigma_j)[p_j], q_i \rangle_{\sout}}{\varrho_i - \sigma_j},
    \end{equation*}
    which proves~\eqref{eq:Erdata}. Similarly for $\rAop^{(i,j)} = e_i^\top\olm\gen\orm e_j$, we have
    \begin{equation}\label{eq: ipderAr}
        \rAop^{(i,j)} =\langle \oout(\varrho_{i}~\id_\ssb - \gen)^{-1}\gen(\sigma_j~\id_\ssb - \gen)^{-1}\oin[p_j] , q_i\rangle_{\sout}.
    \end{equation}
    For $\sigma_j=\varrho_i$ we add and subtract $\sigma_j~\id_X$ to rewrite the first variable as
    \begin{equation*}
    \begin{aligned}
        \oout(\sigma_j~\id_\ssb - \gen)^{-1}&\gen(\sigma_j~\id_\ssb - \gen)^{-1}\oin[p_j] =\\& \oout(\sigma_j~\id_\ssb - \gen)^{-1}(\sigma_j~\id_x - \sigma_j~\id_X + \gen)(\sigma_j~\id_\ssb - \gen)^{-1}\oin[p_j],
    \end{aligned}
    \end{equation*}
    which can be then further manipulated as
    \begin{equation*}
    \begin{aligned}
        \sigma_j~\oout(\sigma_j~\id_\ssb - \gen)^{-1}& (\sigma_j~\id_\ssb - \gen)^{-1}\oin[p_j] +\\ &+\oout(\sigma_j~\id_\ssb - \gen)^{-1}(- \sigma_j~\id_X + \gen)(\sigma_j~\id_\ssb - \gen)^{-1}\oin[p_j]
        \end{aligned}
    \end{equation*}
    Then by \cref{eq: derTF}, the first summand in his last expression becomes
    \begin{equation*}
        \sigma_j~\oout(\sigma_j~\id_\ssb - \gen)^{-1} (\sigma_j~\id_\ssb - \gen)^{-1}\oin[p_j] = -\sigma_j \dfrac{d}{ds}\tf(\sigma_j)[p_j].
    \end{equation*}
    For the second summand we have
    \begin{equation*}
    \begin{aligned}
        \oout(\sigma_j~\id_\ssb - \gen)^{-1}(- \sigma_j~\id_X + \gen)(\sigma_j~\id_\ssb - \gen)^{-1}\oin[p_j] &= -\oout(\sigma_j~\id_\ssb - \gen)^{-1}\oin[p_j]\\
        &= -\tf(\sigma_j)[p_j].
    \end{aligned}
    \end{equation*}
    Hence, \cref{eq: ipderAr} becomes
    \begin{equation*}
    \begin{aligned}
        \rAop^{(i,j)} = -\langle \tf(\sigma_j)[p_j] +\sigma_j \dfrac{d}{ds}\tf(\sigma_j)[p_j], q_i\rangle_{\sout},
    \end{aligned}
    \end{equation*}
    proving \cref{eq:Arrepateddata}. For the case with $\sigma_j\neq \varrho_i$, we use the resolvent identity
    \begin{equation*}
        (\varrho~\id_\ssb - \gen)^{-1}\gen(\sigma~\id_\ssb - \gen)^{-1} = -\dfrac{\varrho(\varrho~\id_\ssb - \gen)^{-1} - \sigma(\sigma~\id_\ssb - \gen)^{-1}}{\varrho - \sigma}
    \end{equation*}
    to get
    \begin{equation*}
    \begin{aligned}
        \rAop^{(i,j)} &=\langle \oout(\varrho_{i}~\id_\ssb - \gen)^{-1}\gen(\sigma_j~\id_\ssb - \gen)^{-1}\oin[p_j] , q_i\rangle_{\sout} \\
                    &= -\left\langle \oout\left(\dfrac{\varrho_i(\varrho_{i}~\id_\ssb - \gen)^{-1} - \sigma_j (\sigma_j~\id_\ssb - \gen)^{-1}}{\varrho_i - \sigma_j}\right)\oin[p_j] , q_i\right\rangle_{\sout} \\
                    &= -\left\langle \dfrac{\varrho_i\oout(\varrho_{i}~\id_\ssb - \gen)^{-1}\oin[p_j] - \sigma_j\oout(\sigma_j~\id_\ssb - \gen)^{-1}\oin[p_j]}{\varrho_i - \sigma_j} , q_i\right\rangle_{\sout}  \\
                    &= -\dfrac{\langle \varrho_i\oout(\varrho_{i}~\id_\ssb - \gen)^{-1}\oin[p_j] , q_i\rangle_{\sout} - \langle \sigma_j\oout(\sigma_j~\id_\ssb - \gen)^{-1}\oin[p_j], q_i \rangle_{\sout}}{\varrho_i - \sigma_j} \\ 
                    &= -\dfrac{\varrho_i\langle \tf(\varrho_i)[p_j] , q_i\rangle_{\sout} - \sigma_j\langle \tf(\sigma_j)[p_j], q_i \rangle_{\sout}}{\varrho_i - \sigma_j}.
    \end{aligned}
    \end{equation*}
    Again, using the Hilbert adjoint property we obtain
    \begin{equation*}
        \rAop^{(i,j)} = -\dfrac{\varrho_i\langle p_j , \tf(\varrho_i)^\dagger[q_i]\rangle_{\sin} - \sigma_j\langle \tf(\sigma_j)[p_j], q_i \rangle_{\sout}}{\varrho_i - \sigma_j},
    \end{equation*}
    which proves~\eqref{eq:Ardata}.
    Given an $f\in \sin$, for $\rBop^{(i)}[f] = e_i^\top\olm\oin[f]$ we have
    \begin{equation*}
    \begin{aligned}
        \rBop^{(i)}[f] &= e_i^\top\olm\oin[f] = \langle(\varrho_{i}~\id_\ssb - \gen)^{-*}\oout^*[q_{i}],\oin[f]\rangle_{(\ssb^*,\ssb)} \\
        &= \langle\oin^*(\varrho_{i}~\id_\ssb - \gen)^{-*}\oout^*[q_{i}],f\rangle_{(\sin^*,\sin)}.
    \end{aligned}
    \end{equation*}
    By Riesz representation theorem we can write this via the inner product on $\sin$ as
    \begin{equation*}
    \begin{aligned}
        \langle\oin^*(\varrho_{i}~\id_\ssb - \gen)^{-*}\oout^*[q_{i}],f\rangle_{(\sin^*,\sin)} &= \langle\left(\oout(\varrho_{i}~\id_\ssb - \gen)^{-1}\oin\right)^*[q_{i}],f\rangle_{(\sin^*,\sin)} \\
        &= \left\langle f,\left(\oout(\varrho_{i}~\id_\ssb - \gen)^{-1}\oin\right)^\dagger[q_{i}]\right\rangle_{\sin}.
        \end{aligned}
    \end{equation*}
    Since $\tf(s) = \oout(s~\id_\ssb - \gen)^{-1}\oin$, we can simplify this as
    \begin{equation*}
    \begin{aligned}
        \rBop^{(i)}[f] &= \left\langle f,\left(\oin(\varrho_{i}~\id_\ssb - \gen)^{-1}\oout\right)^\dagger[q_{i}]\right\rangle_{\sin} 
        = \left\langle f,\tf(\varrho_i)^\dagger[q_{i}]\right\rangle_{\sin},
    \end{aligned}
    \end{equation*}
        proving the first expression in~\eqref{eq:BrCrdata}. Finally, for $\rCop^{(j)} = \oout\orm[e_j]$ we use the definition of $\orm$ to rewrite it as
    \begin{equation*}
    \begin{aligned}
        \rCop^{(j)} = \oout\orm [e_j] = \oout (\sigma_{j}~\id_\ssb - \gen)^{-1}\oin[p_j] = \tf(\sigma_j)[p_j],
    \end{aligned}
    \end{equation*}
    completing the proof.
\end{proof}
Thus, we established, for the first time, the Loewner framework  for LTI systems where not only the state space but also the input-output spaces are infinite dimensional, enabling a data-driven realization of an interpolatory reduced order system directly from transfer function evaluations without requiring access to the underlying operators of the full model.

%%%%%%%%%%%%%%%%%%%%%%%%%%%%%%%%%%%%%%%%%%%%%%%%%%%%%%%%%%%%%%%%%%%%%%%%%%%%%%%%
\subsection{An illustrative example}\label{ss: infEx}
In this section, we will demonstrate our theoretical results for a 2D heat equation on the unit square. Our presentation is inspired from \cite[Example 8.2]{EngelNagel} as well as \cite{CurM09} where transfer functions for several PDEs have been derived explicitly. 

\subsubsection{Setup and the associated transfer function}
Consider the following controlled heat equation on the domain $\Omega = (0,1)\times (0,1)$. Denote the temperature at position $z = \bbm z_1 & z_2 \ebm^\top \in \Omega$ and time $t\geq 0$ by $w(z,t)$. Given an initial temperature profile $w_0(\cdot)$ and a fixed $r>0$, let us consider
\begin{equation}\label{eq: exPDE}
\begin{aligned}
    \dfrac{\partial w}{\partial t}(z,t) &= \dfrac{\partial^2w}{\partial z_1^2}(z,t) + \dfrac{\partial^2w}{\partial z_2^2}(z,t) + \mathbf{1}_{\omega_{\mathrm{con}}}~u(z,t), \\
    w(z,t) &=  0, && \hspace{-2cm} \text{for } z\in \partial\Omega,~t\geq 0, \\
    w(z,0) &= w_0(z), && \hspace{-2cm} \text{for } z\in \Omega,\\
    y(z,t) &= \mathbf{1}^*_{\omega_{\mathrm{obs}}}~w(z,t),&& \hspace{-2cm} \text{for } z \in \Omega,~t\geq 0.
\end{aligned}
\end{equation}
Here, for $K\subset \Omega$ the operator $\mathbf{1}_{K} \in \mathcal{L}(L_2(K),L_2(\Omega))$ is defined as
\begin{align*}
 (\mathbf{1}_{K}(f))(z) =\begin{cases} f(z), & z \in K,  \\ 0, & z\in \Omega
\backslash K \end{cases}
\end{align*}
and $\mathbf{1}^*_{K}\in \mathcal{L}(L_2(\Omega),L_2(K))$ is the restriction operator from $\Omega$ to $K$.
The control and observation domains $\omega_{\mathrm{con}}\subseteq\Omega,\omega_{\mathrm{obs}}\subseteq \Omega$ are defined as
\begin{equation*}
    \omega_{\mathrm{con}} = [z_{\mathrm{con}}-r,z_{\mathrm{con}}+r]\times[z_{\mathrm{con}}-r,z_{\mathrm{con}}+r], \quad\omega_{\mathrm{obs}} = [z_{\mathrm{obs}}-r,z_{\mathrm{obs}}+r]\times[z_{\mathrm{obs}}-r,z_{\mathrm{obs}}+r].
\end{equation*}
Let $\sin = L_2(\omega_\mathrm{con})$ and $\sout = L_2(\omega_\mathrm{obs})$. Then, the system can be described by a cACP with a linear operator $\gen$ on $X$ defined as
\begin{equation*}
\gen[z] = \Delta z,~\dom(\gen) := \left\{z\in H_2(\Omega) \mid z_{|_{\partial\Omega}} = 0 \right\}.
\end{equation*}
Then $\oin \colon\sin \to \ssb$ and $\oout \colon \ssb \to \sout$ become
\begin{equation*}
    \oin[v](z) := (\mathbf{1}_{\omega_{\mathrm{con}}}(v))(z),~\text{for} ~v\in \sin~\text{and}~\oout[w](z) := (\mathbf{1}_{\omega_{\mathrm{obs}}}^*(w))(z),~\text{for} ~w\in \ssb.
\end{equation*}
Let $e_n\colon(0,1)\to \R$ be defined as
\begin{equation*}
    e_n(t) := \sqrt{2}~\sine(n\pi t),~\text{for } n\geq 1.
\end{equation*}
Then, $\{e_n\}_{n=1}^\infty$ is orthonormal in $L_2(0,1)$ and satisfies
\begin{equation*}
    -\dfrac{\mathrm{d}^2}{\mathrm{d}t^2}e_n = (n\pi)^2e_n,~ e_n(0) = e_n(1) = 0.
\end{equation*}
Hence, the set of functions $\{\phi_{nm}\}_{n,m=1}^\infty$
\begin{equation*}
    \phi_{nm}(z) = e_n(z_1)e_m(z_2),~n,m\geq 1
\end{equation*}
is an orthonormal basis for $L_2(\Omega)$. Moreover, we have
\begin{equation*}
    \gen[\phi_{nm}] = \Delta \phi_{nm} = \lambda_{nm}\phi_{nm},\quad \lambda_{nm} := -\pi^2(n^2+m^2).
\end{equation*}
Hence, for $\Re(s) > 0$ we have
\begin{equation}\label{eq: rat_res}
    (s~\id_\ssb - \gen)^{-1}[f] = \sum_{n,m \geq 1}\dfrac{\langle f,\phi_{nm}\rangle}{s-\lambda_{nm}}\phi_{nm}.
\end{equation}
The transfer function can then be written as
\begin{equation}\label{eq: exTF}
    \oout(s~\id_\ssb - \gen)^{-1}\oin[v](z) = \mathbf{1}_{K{(z_{\mathrm{con}},\ell)}}(z)\sum_{n,m\geq 1} \dfrac{\langle \mathbf{1}_{K{(z_{\mathrm{con}},\ell)}}~v,\phi_{nm}\rangle}{s- \lambda_{nm}}\phi_{nm}(z).
\end{equation}

\subsubsection{Constructing the reduced order model}

We fix some tangential directions $q_i\in \Y,p_j\in \U$, interpolation points $\varrho_i, \sigma_j$ such that $\sigma_j\neq \varrho_i$ for $i,j=1,\dots,r$ and want to construct a reduced order model such that its transfer function $\rtf$ interpolates $\tf$ at those points and directions. To do this, we will use \cref{thm: intPG}. So, we will pick $\olm,\orm$ as in \cref{eq: fixVrWr} and construct our reduced operators $\rEop,\rAop,\rBop,\rCop$ as described in \cref{thm: Loewner}. Together with \cref{eq: exTF}, for $\rEop$, this yields
\begin{equation*}
\begin{aligned}
    \rEop^{(i,j)} &= -\dfrac{\langle \tf(\varrho_i)[p_j],q_i\rangle_{\sout}-\langle \tf(\sigma_j)[p_j], q_i\rangle_{\sout}}{\varrho_i - \sigma_j} 
                \\
                &= -\dfrac{1}{\varrho_i - \sigma_j}\left\langle \mathbf{1}_{\omega_{\mathrm{obs}}}(\cdot)\sum_{n,m\geq 1} \dfrac{\langle \mathbf{1}_{\omega_{\mathrm{con}}}~p_j,\phi_{nm}\rangle}{\varrho_i- \lambda_{nm}}\phi_{nm}(\cdot),q_i\right\rangle 
                \\
                & \quad +\dfrac{1}{\varrho_i - \sigma_j} \left\langle \mathbf{1}_{\omega_{\mathrm{obs}}}(\cdot)\sum_{n,m\geq 1} \dfrac{\langle \mathbf{1}_{\omega_{\mathrm{con}}}~p_j,\phi_{nm}\rangle}{\sigma_j- \lambda_{nm}}\phi_{nm}(\cdot),q_i\right\rangle.
\end{aligned}
\end{equation*}
We can take the inner sum out and use inner product properties to rewrite this as
\begin{equation*}
\begin{aligned}
    \rEop^{(i,j)} &= \dfrac{1}{\varrho_i - \sigma_j}\sum_{n,m\geq 1}\left\langle\mathbf{1}_{\omega_{\mathrm{obs}}}(\cdot) \dfrac{\langle \mathbf{1}_{\omega_{\mathrm{con}}}~p_j,\phi_{nm}\rangle}{\sigma_j- \lambda_{nm}}\phi_{nm}(\cdot),q_i\right\rangle \\ &\quad- \dfrac{1}{\varrho_i - \sigma_j}\sum_{n,m\geq 1}\left\langle\mathbf{1}_{\omega_{\mathrm{obs}}}(\cdot) \dfrac{\langle \mathbf{1}_{\omega_{\mathrm{con}}}~p_j,\phi_{nm}\rangle}{\varrho_i- \lambda_{nm}}\phi_{nm}(\cdot),q_i\right\rangle
                \\
                &= \dfrac{1}{\varrho_i - \sigma_j}\sum_{n,m\geq 1}\langle \mathbf{1}_{\omega_{\mathrm{con}}}~p_j,\phi_{nm}\rangle\left(\dfrac{1}{\sigma_j- \lambda_{nm}}- \dfrac{1}{\varrho_i- \lambda_{nm}}\right)\left\langle\mathbf{1}_{\omega_{\mathrm{obs}}}~\phi_{nm},q_i\right\rangle
                \\
                &= \dfrac{1}{\varrho_i - \sigma_j}\sum_{n,m\geq 1}\dfrac{\varrho_i - \sigma_j}{(\sigma_j- \lambda_{nm})(\varrho_i- \lambda_{nm})}\langle \mathbf{1}_{\omega_{\mathrm{con}}}~p_j,\phi_{nm}\rangle\left\langle\mathbf{1}_{\omega_{\mathrm{obs}}}~\phi_{nm},q_i\right\rangle
                \\
                &= \sum_{n,m\geq 1}\dfrac{\langle \mathbf{1}_{\omega_{\mathrm{con}}}~p_j,\phi_{nm}\rangle\left\langle\mathbf{1}_{\omega_{\mathrm{obs}}}~\phi_{nm},q_i\right\rangle}{(\sigma_j- \lambda_{nm})(\varrho_i- \lambda_{nm})}.
\end{aligned}
\end{equation*}
One can similarly calculate $\rAop$. Recall that \cref{thm: Loewner} together with \cref{eq: exTF} yields
\begin{equation*}
\begin{aligned}
    \rAop^{(i,j)} &= -\dfrac{\varrho_i\langle \tf(\varrho_i)[p_j],q_i\rangle_{\sout}-\sigma_j\langle \tf(\sigma_j)[p_j], q_i\rangle_{\sout}}{\varrho_i - \sigma_j} 
                \\
                &= -\dfrac{\varrho_i}{\varrho_i - \sigma_j}\left\langle \mathbf{1}_{\omega_{\mathrm{obs}}}(\cdot)\sum_{n,m\geq 1} \dfrac{\langle \mathbf{1}_{\omega_{\mathrm{con}}}~p_j,\phi_{nm}\rangle}{\varrho_i- \lambda_{nm}}\phi_{nm}(\cdot),q_i\right\rangle 
                \\
                & \quad +\dfrac{\sigma_j}{\varrho_i - \sigma_j} \left\langle \mathbf{1}_{\omega_{\mathrm{obs}}}(\cdot)\sum_{n,m\geq 1} \dfrac{\langle \mathbf{1}_{\omega_{\mathrm{con}}}~p_j,\phi_{nm}\rangle}{\sigma_j- \lambda_{nm}}\phi_{nm}(\cdot),q_i\right\rangle.
\end{aligned}
\end{equation*}
This then simplifies to
\begin{equation*}
\begin{aligned}
    \rAop^{(i,j)} &= \dfrac{1}{\varrho_i - \sigma_j}\sum_{n,m\geq 1}\langle \mathbf{1}_{\omega_{\mathrm{con}}}~p_j,\phi_{nm}\rangle\left(\dfrac{\sigma_j}{\sigma_j- \lambda_{nm}}- \dfrac{\varrho_i}{\varrho_i- \lambda_{nm}}\right)\left\langle\mathbf{1}_{\omega_{\mathrm{obs}}}~\phi_{nm},q_i\right\rangle 
                \\
                &= \dfrac{1}{\varrho_i - \sigma_j}\sum_{n,m\geq 1}\langle \mathbf{1}_{\omega_{\mathrm{con}}}~p_j,\phi_{nm}\rangle\dfrac{\sigma_j\varrho_i - \sigma_j\lambda_{nm} - \sigma_j\varrho_i + \varrho_i\lambda_{nm}}{(\sigma_j- \lambda_{nm})(\varrho_i- \lambda_{nm})}\left\langle\mathbf{1}_{\omega_{\mathrm{obs}}}~\phi_{nm},q_i\right\rangle 
                \\
                &= \dfrac{1}{\varrho_i - \sigma_j}\sum_{n,m\geq 1}\langle \mathbf{1}_{\omega_{\mathrm{con}}}~p_j,\phi_{nm}\rangle\dfrac{(\varrho_i- \sigma_j)\lambda_{nm}}{(\sigma_j- \lambda_{nm})(\varrho_i- \lambda_{nm})}\left\langle\mathbf{1}_{\omega_{\mathrm{obs}}}~\phi_{nm},q_i\right\rangle
                \\
                &= \sum_{n,m\geq 1}\dfrac{\lambda_{nm}\langle \mathbf{1}_{\omega_{\mathrm{con}}}~p_j,\phi_{nm}\rangle\left\langle\mathbf{1}_{\omega_{\mathrm{obs}}}~\phi_{nm},q_i\right\rangle}{(\sigma_j- \lambda_{nm})(\varrho_i- \lambda_{nm})}.
\end{aligned}
\end{equation*}
For $\rBop$ and a given $f\in \sin$, \cref{thm: Loewner} together with \cref{eq: exTF} yields
\begin{equation*}
\begin{aligned}
    \rBop^{(i)}[f]=\langle \tf(\varrho_i)[f],q_i\rangle_{\sout} = \langle \mathbf{1}_{\omega_{\mathrm{obs}}}\sum_{n,m\geq 1} \dfrac{\langle \mathbf{1}_{\omega_{\mathrm{con}}}~f,\phi_{nm}\rangle}{\varrho_i- \lambda_{nm}}\phi_{nm},q_i\rangle.
\end{aligned}
\end{equation*}
By taking the sum out, this can be simplified to
\begin{equation*}
\begin{aligned}
    \rBop^{(i)}[f]=\sum_{n,m\geq 1} \dfrac{\langle \mathbf{1}_{\omega_{\mathrm{con}}}~f,\phi_{nm}\rangle \langle \mathbf{1}_{\omega_{\mathrm{obs}}} ~\phi_{nm},q_i\rangle}{\varrho_i- \lambda_{nm}}.
\end{aligned}
\end{equation*}
Lastly $\rCop$ can be written coordinate-wise as
\begin{equation*}
    \rCop^{(j)} = \tf(\sigma_j)[p_j] = \mathbf{1}_{\omega_{\mathrm{obs}}}\sum_{n,m\geq 1} \dfrac{\langle \mathbf{1}_{\omega_{\mathrm{con}}}~p_j,\phi_{nm}\rangle}{\sigma_j- \lambda_{nm}}\phi_{nm}.
\end{equation*}
Thus, we have shown how to capture the operators for the reduced cACP directly, without constructing $\olm,\orm$ explicitly.

%        %
% %%%%%%%%%%%%%%%%%%%%%%%%%%%%%%%%%%%%%%%%%%%%%%%%%%%%%%%%%%%%%%%%%%%%%%%%%%%%%%%%

\section{$\mathcal{H}_2$ optimal model reduction}\label{sect: infH2opt}

In \cref{sect: IntMOR} we have developed an interpolatory model reduction framework for controlled abstract Cauchy problems where the state space is a Banach space and the input, output spaces are Hilbert spaces. Specifically, we have defined reduced order models for \cref{eq: cACP} via Petrov-Galerkin projection and shown that we can enforce interpolation of the original transfer function for some specific left and right modeling subspaces. However, we have not discussed how to pick the interpolation data in an optimal way. In this section, we answer this question by extending the interpolatory $\mathcal{H}_2$ optimal approximation theory to an infinite dimensional setting.

\subsection{Revisiting optimal $\mathcal{H}_2$ model reduction for finite dimensional input-output systems}\label{ss: finH2opt}
We first briefly revisit the $\mathcal{H}_2$ optimal approximation theory
for finite dimensional LTI systems, which will help motivate the 
infinite dimensional extension in the next section. Below, we mainly follow
the discussion from~\cite{finH2}.  

Recall the model reduction problem of finite dimensional LTI systems discussed in~\Cref{sec: intro} and the corresponding output error bound~\eqref{eq:yminusyr} in terms of the distance between the full model and reduced model transfer functions $\mtf$ and $\rmtf$. 
In~\eqref{eq:yminusyr}, for $u\in L_2(\R_+;\R^m)$, we will choose to work with the $L_\infty(\R_+;\R^p)$ norm of the output error, namely
\begin{equation*} 
    \|y-y_r\|_{L_\infty} = \max_{t\geq 0}\|y(t)- y_r(t)\|_\infty.
\end{equation*}
This norm can be bounded as \cite[Chapter 2.1]{finH2}
\begin{equation}\label{eq: H2_ineq}
    \|y-y_r\|_{L_\infty} \leq \left(\dfrac{1}{2\pi}\int_{\omega = -\infty}^\infty\|\bb{\mtf}(i \omega) - \bb{\rmtf}(i\omega)\|_F^2~\d \omega\right)^{1/2}\|u\|_{L_2}.
\end{equation}
Hence, to get a reduced system whose output is close to that of the full model with respect to the $L_\infty$ distance, one should minimize the $L_2$ distance between $\mtf$ and $\rmtf$ on the imaginary axis. 
This automatically leads to the  $\mathcal{H}_2$ norm of a dynamical system. More precisely, the $\mathcal{H}_2$ norm of an asymptotically stable LTI system~\eqref{eq:finiteLTIsys} with the transfer function $\mtf$~\eqref{eq:Gfinite} is given as 
\begin{equation}\label{eq: finH2_norm}
   \| \mtf \|_{\mathcal{H}_2^{p\times m}} = \left(\dfrac{1}{2\pi}\int_{\omega = -\infty}^\infty\|\mtf(i \omega)\|_F^2~\d \omega\right)^{1/2},
\end{equation}
where $\| \mathbf{M}\|_F^2 = \tr(\overline{\mathbf{M}}\mathbf{M}^\top)$ and $\overline{\mathbf{M}}$ denotes the element-wise conjugation of the matrix $\mathbf{M}\in \C^{p\times m}$. For single-input ($m=1$) and/or single-output ($p=1$) systems, 
the $\mathcal{H}_2$ norm is the $(2,\infty)$ induced norm of the underlying convolution operator. 

The $\mathcal{H}_2$ norm results from an underlying Hilbert space structure. 
Let $\mtf$ be a $p \times m$ matrix-valued function with components that are analytic in the open right half-plane $\{\,z \in \mathbb{C} : \operatorname{Re}(z) > 0\,\}$, 
and satisfying
\begin{equation}
    \sup_{x>0}\int_{-\infty}^{\infty}\|\mtf(x+ i y)\|_{F}^{2}\, \mathrm{d}y < \infty.
\end{equation}
The Hardy space $\mathcal{H}_2^{p \times m}$ is the vector space of all such matrix-valued functions. It is a Hilbert space endowed with the  inner product
\begin{equation*}
     \langle \mathbf{H},\mtf \rangle_{\mathcal{H}_2^{p\times m}} := \dfrac{1}{2\pi}\int_{\omega = -\infty}^\infty \tr\left(\overline{\mtf(i \omega)}\mathbf{H} (i\omega)^\top\right)~\d \omega,
\end{equation*}
where 
$\mtf$ and $\mathbf{H}$ are transfer functions of associated
asymptotically stable dynamical systems with $m$ inputs and $p$ outputs. 

Now inspired by~\eqref{eq: H2_ineq}, given the original LTI system with the transfer function $\mtf$, the goal of $\mathcal{H}_2$ optimal approximation problem is to find a reduced model~\eqref{eq:finiteLTIsysred} with the reduced degree-$r$ rational function $\rmtf$ that minimizes the $\mathcal{H}_2^{p\times m}$ distance $\|\mtf -\rmtf\|_{\mathcal{H}_2^{p\times m}}$. The following theorem gives the interpolatory necessary conditions for $\rmtf$ to be $\mathcal{H}_2$ optimal.

\begin{theorem}[$\mathcal{H}_2$ optimality conditions for LTI systems \cref{eq:finiteLTIsys}]\label{thm: fin_dim_h2_opt}\cite[Theorem 5.1.1]{finH2}
    Let $\rmtf$ be a locally optimal $r^{\text{th}}$ order approximant of $\mtf$ with respect to the $\mathcal{H}_2^{p\times m}$ norm and let $\rmtf$ have the pole-residue form \cref{eq: fin_dim_pr}. Then $\rmtf$ is a tangential interpolant of $\mtf$, i.e., for $i = 1,\dots,r$
    \begin{equation*}
    \begin{aligned}
        \rmtf(-\overline{\lambda_i})\mathbf{b}_i &= \mtf(-\overline{\lambda_i})\mathbf{b}_i, \\
        \mathbf{c}_i^*\rmtf(-\overline{\lambda_i}) &= \mathbf{c}_i^*\mtf(-\overline{\lambda_i}), \\
        \mathbf{c}_i^*\left(\dfrac{\mathrm{d}}{\mathrm{d}s}\rmtf(-\overline{\lambda_i})\right)\mathbf{b}_i &= \mathbf{c}_i^*\left( \dfrac{\mathrm{d}}{\mathrm{d}s}\mtf(-\overline{\lambda_i})\right)\mathbf{b}_i.
    \end{aligned}
    \end{equation*}
\end{theorem}
The following result, a simplified version of \cite[Lemma 2.1.4]{finH2}, is essential in proving \cref{thm: fin_dim_h2_opt}. 
\begin{lemma} \label{lemma:rank1finite}
    Let $\mtf$ be the transfer function of a stable LTI system of the form \cref{eq:finiteLTIsys} whose poles are contained in the open left half-plane. For a fixed $\lambda\in \C_{-}$, $\mathbf{b}\in \C^m$ and $\mathbf{c}\in \C^p$, define $\mathbf{F}\in \mathcal{H}_2^{p\times m}$ as
    \begin{equation*}
        \mathbf{F}(s) := \dfrac{1}{(s - \lambda)}\mathbf{c}\mathbf{b}^*.
    \end{equation*}
    Then one can evaluate the $\mathcal{H}_2^{p\times m}$ inner product of $\mtf$ and $\mathbf{F}$ as 
    \begin{equation*}
        \langle\mathbf{F},\mtf\rangle_{\mathcal{H}_2^{p\times m}} = \mathbf{c}^*{\mtf(\overline{-\lambda})}\mathbf{b}.
    \end{equation*}
\end{lemma}

Our goal in this section is to extend \cref{thm: fin_dim_h2_opt} to a cACP of the form \cref{eq: cACP}. Specifically, we will define a Hardy space $\mathcal{H}_2(\sin,\sout)$ when both input space $\sin$ and output space $\sout$ are \emph{separable} Hilbert spaces. Then, we will identify the \emph{transfer function} $\tf$ of the cACP \cref{eq: cACP} as an element of this space and show that we can extend \Cref{thm: fin_dim_h2_opt} to this case under certain assumptions.

\subsection{$\mathcal{H}_2$ optimal model reduction for infinite dimensional systems}\label{ss: infH2Opt}
The Hardy space we will consider is defined as
\begin{equation}\label{eq: H2_space}
    \mathcal{H}_2(\sin,\sout) := \left\{\tf\colon\C_+\to \mathcal{N}_2(\sin,\sout)~\mid~ \tf\text{ holomorphic on }\C_+\right\}
\end{equation}
with the norm
\begin{equation*}
    \|\tf\|_{\mathcal{H}_2(\sin,\sout)}^2 = \sup_{\alpha>0}~\frac{1}{2\pi} \int_{-\infty}^\infty \|\tf(\alpha+i\omega)\|_{\mathcal{N}_2(\sin,\sout)}^2 \,\mathrm{d}\omega,
\end{equation*}
where $\mathcal{N}_2$ norm is the Hilbert-Schmidt norm defined as
\begin{equation}\label{def: trace}
    \|\mathcal{S}\|_{\mathcal{N}_2(\sin,\sout)}^2 = \sum_{i=1}^\infty\|\mathcal{S}[p_i]\|_{\sout}^2 = \tr(\mathcal{S}^\dagger\mathcal{S}),
\end{equation}
for any $\mathcal{S}\in \mathcal{N}_2(\sin,\sout)$ and orthonormal basis $\{p_i\}_{i=1}^\infty \subset \sin$. This norm is induced from the $\mathcal{N}_2$ inner product defined as,
\begin{equation}\label{def: traceip}
    \langle \mathcal{S},\mathcal{T} \rangle_{\mathcal{N}_2(\sin,\sout)} = \sum_{i=1}^\infty\langle\mathcal{S}[p_i],\T[p_i]\rangle_{\sout} = \tr(\mathcal{T}^\dagger\mathcal{S}).
\end{equation}
One can show that the $\mathcal{H}_2$ norm is given as
\begin{equation*}
    \|\tf\|_{\mathcal{H}_2(\sin,\sout)}^2 = \left(\frac{1}{2\pi} \int_{-\infty}^\infty \|\tf(i\omega)\|_{\mathcal{N}_2(\sin,\sout)}^2 \,\mathrm{d}\omega\right).
\end{equation*}
Indeed, since $\tf$ is holomorphic on $\C_+$, $s\mapsto \|\tf(s)\|_{\mathcal{N}_2(\sin,\sout)}$ is subharmonic as a function of $s\in \C_+$ \cite[Theorem 3.13.1]{HillePhilips}. Thus, the maximum modulus principle applies, and one can rewrite the $\mathcal{H}_2$ norm as
\begin{equation}\label{eq: H2_norm}
     \|\tf\|_{\mathcal{H}_2(\sin,\sout)}^2 = \sup_{\alpha>0}~\frac{1}{2\pi} \int_{-\infty}^\infty \|\tf(\alpha+i\omega)\|_{\mathcal{N}_2(\sin,\sout)}^2 \,\mathrm{d}\omega = \frac{1}{2\pi} \int_{-\infty}^\infty \|\tf(i\omega)\|_{\mathcal{N}_2(\sin,\sout)}^2 \,\mathrm{d}\omega.
\end{equation}
Hence, the inner product induced by the $\mathcal{H}_2$ norm can be written as 
\begin{equation*}
    \langle \mathcal{F},\tf \rangle_{\mathcal{H}_2(\sin,\sout)} = \dfrac{1}{2\pi}\int_{-\infty}^\infty\langle \mathcal{F}(i\omega), \tf(i\omega)\rangle_{\mathcal{N}_2(\sin,\sout)}\,\mathrm{d}\omega.
\end{equation*}
Then by \cref{def: traceip}, this becomes
\begin{equation}\label{eq: H2_ip}
    \langle \mathcal{F}, \tf \rangle_{\mathcal{H}_2(\sin,\sout)} = \dfrac{1}{2\pi}\int_{-\infty}^\infty \tr(\tf(i\omega)^\dagger\mathcal{F}(i\omega))\,\mathrm{d}\omega.
\end{equation}

The following inner product evaluation will help us get the $\mathcal{H}_2$ optimality conditions as in the finite dimensional case. It generalizes 
\Cref{lemma:rank1finite} to the infinite dimensional case. 
\begin{theorem}[Rank-1 outer products]\label{thm: rank-1_prod}
    Given $\tf \in \mathcal{H}_2(\sin,\sout)$, fix $p\in \sin$, $q\in \sout$, $\lambda\in \C_-$ and let $\mathcal{F}\in \mathcal{H}_2(\sin,\sout)$ be
    \begin{equation*}
        \mathcal{F}(s)[f] := \dfrac{\langle f, p\rangle_{\sin}}{s-\lambda}\cdot q~,\quad \text{for } f\in \sin.
    \end{equation*}
    Then we have
    \begin{equation*}
        \langle \mathcal{F},\tf\rangle_{\mathcal{H}_2(\sin,\sout)} = \langle q,\tf(-\overline{\lambda})[p]\rangle_{\sout}.
    \end{equation*}
\end{theorem}
\begin{proof}
    Let $\tf^\dagger$ be the \emph{adjoint function} defined as $\tf^\dagger(s) = \tf(\overline{s})^\dagger$. This definition can be motivated by considering the Taylor expansion of $\tf$ \cite[Theorem 3.11.4]{HillePhilips} and taking the adjoint for each element of the sum. Moreover, one can show that for holomorphic $\tf$, $\tf^\dagger$ is also a holomorphic function because it has the Taylor expansion with adjoint elements of the sum. 
    
    Since $\tf(i\omega)^\dagger =\tf^\dagger(\overline{i\omega}) = \tf^\dagger(-i\omega)$ for $\omega\in \R$, \cref{eq: H2_ip} becomes
    \begin{equation*}
        \langle \mathcal{F},\tf\rangle_{\mathcal{H}_2(\sin,\sout)} = \dfrac{1}{2\pi}\int_{-\infty}^\infty \text{trace}( \tf^\dagger(-i\omega)\mathcal{F}(i\omega))\,\mathrm{d}\omega.
    \end{equation*}
    Because $\tf$ is holomorphic, $\tf^\dagger$ is holomorphic on its domain. This implies that $\tf^\dagger(-\cdot)$ is holomorphic on its respective domain. Thus, the scalar valued $s \mapsto \text{trace}(\tf^\dagger(-s)\mathcal{F}(s))$ is holomorphic on its domain. 
    From this point on, the proof manipulates the scalar-valued function $s \mapsto \text{trace}(\tf^\dagger(-s)\mathcal{F}(s))$ to get the result as in the finite dimensional case \cite[Theorem 2.1.4]{finH2}. Since now this remaining part of the proof is identical to the finite dimensional case, we refer to \Cref{sec:Thm8Proof} for the details.
\end{proof}

\Cref{thm: rank-1_prod} suggests that for any $G\in \mathcal{H}_2(\sin,\sout)$ we can write the inner product of $G$ with a rank-1 operator by simply evaluating $G$ at the mirror image of the pole $-\overline{\lambda}$ of that rank-1 operator. Since the inner product is linear, this can be extended to any $\rtf\in \mathcal{H}_2(\sin,\sout)$, which has $r\in\N$ semi-simple poles. In this case, similar to~\eqref{eq: fin_dim_pr} in the finite dimensional case, $\rtf$ admits the  pole-residue form
\begin{equation*}
    \rtf(s) = \sum_{i=1}^r\dfrac{\langle\cdot,b_i\rangle_{\sin}}{s- \lambda_i}c_i
\end{equation*}
for some $p_i\in \sin,q_i\in \sout$ with $\lambda_i\in \C$ being its semi-simple poles. Hence we get
\begin{equation*}
    \langle \tf, \rtf\rangle_{\mathcal{H}_2(\sin,\sout)} = \sum_{i=1}^r \left\langle \tf,\dfrac{\langle\cdot,b_i\rangle_{\sin}}{s- \lambda_i}c_i\right\rangle_{\mathcal{H}_2(\sin,\sout)} = \sum_{i=1}^r\langle c_i,\tf(-\overline{\lambda_i})[b_i]\rangle_{\sout}.
\end{equation*}
With this result, we have now established all the ingredients to prove $\mathcal{H}_2$ optimality for the infinite dimensional LTI systems of the form \cref{eq: cACP}, where the state space $\ssb$ is a Banach space and the input and output spaces $\sin,\sout$ are Hilbert spaces. Indeed, the following theorem generalizes \cref{thm: fin_dim_h2_opt} to our infinite dimensional setting.

\begin{theorem}[Tangential conditions for $\mathcal{H}_2$ optimality]\label{thm: H2opt}
    Let 
    \begin{equation*} \label{eq:optGr}
    \rtf(s) = \sum_{i=1}^r\dfrac{\langle\cdot,b_i\rangle_{\sin}}{s- \lambda_i}c_i
\end{equation*}
    be an optimal $r$-th order approximation of $\tf$ with respect to the $\mathcal{H}_2$ norm \cref{eq: H2_norm}. Then, $\rtf$ is a tangential interpolant of $\tf$, i.e., for any $i = 1,2,\dots,r$ it satisfies
    \begin{align}
        \tf(-\overline{\lambda_i})[b_i] &= \rtf(-\overline{\lambda_i})[b_i],\label{eq: lint} \\
        \tf(-\overline{\lambda_i})^*[c_i] &= \rtf(-\overline{\lambda_i})^*[c_i], \label{eq: rint}\\
        \left\langle \dfrac{\mathrm{d}}{\mathrm{d}s}\tf(-\overline{\lambda_i})[b_i], c_i \right\rangle_{\sout} &= \left\langle \dfrac{\mathrm{d}}{\mathrm{d}s}\rtf(-\overline{\lambda_i})[b_i],c_i\right\rangle_{\sout}.\label{eq: dint}
    \end{align}
\end{theorem}
\begin{proof}
The proof of this theorem can be found in \Cref{sec:Thm9Proof}. While it structurally parallels the finite dimensional proof \cite[Theorem 5.1.1]{finH2}, the infinite dimensional setting introduces some additional complexity. The key ingredient that makes the generalization possible is \cref{thm: rank-1_prod}, which, together with the structure of the $\mathcal{H}_2(\sin,\sout)$ space \cref{eq: H2_space}, provides the precise functional-analytic foundation needed to carry the argument through in this setting.
\end{proof}

\begin{remark}
    Similarly as in \cref{thm: intPG}, the second assertion \cref{eq: rint} can be equivalently written with the Hilbert adjoint as
    \begin{equation}
        \tf(-\overline{\lambda_i})^\dagger[c_i] = \rtf(-\overline{\lambda_i})^\dagger[c_i].
    \end{equation}
\end{remark}

Note that \cref{thm: H2opt} shows that the $\mathcal{H}_2$ optimality conditions still enforce interpolation even when $\mathcal{A}$ in \cref{eq: cACP} has continuous spectrum, that is, when $\sigma(\mathcal{A})$ contains a continuum or $\mathcal{U},\mathcal{Y}$ are infinite dimensional. Thus, under mild assumptions on the underlying cACP \cref{eq: cACP}, the interpolatory $\mathcal{H}_2$ optimal model reduction framework extends to the infinite dimensional setting. This opens the possibility of generalizing algorithms for computing $\mathcal{H}_2$ optimal reduced models to infinite dimensional systems, and suggests that such methods can be formulated at the PDE level rather than only after discretization. Such an approach may offer advantages in both computational efficiency and accuracy.

%%%%%%%%%%%%%%%%%%%%%%%%%%%%%%%%%%%%%%%%%%%%%%%%%%%%%%%%%%%%%%%%%%%%%%%%%%%%%%%%
\section{Conclusions and future work}\label{sect: concl}
In this paper we have discussed how we can interpret an infinite dimensional LTI system when the state space is a Banach space and the input and output spaces are Hilbert spaces. Then we have discussed the model reduction problem for this setting and focused on solving this problem with Petrov-Galerkin projections. To relate this to the interpolatory model reduction method, we have picked these interpolatory trial and test subspaces and shown that a Petrov-Galerkin projection with these subspaces results in the transfer function of the reduced model to interpolate the transfer function of the original infinite dimensional system. Moreover, we have rewritten this as a Sylvester-like operator equation and shown that the reduced operators can also be written via Loewner matrices. We have shown these results for a 2 dimensional heat equation example by writing the reduced operators explicitly. We then moved to the $\mathcal{H}_2$ optimality problem and shown that similar $\mathcal{H}_2$ optimality conditions also hold in this infinite dimensional setting.

%%%%%%%%%%%%%%%%%%%%%%%%%%%%%%%%%%%%%%%%%%%%%%%%%%%%%%%%%%%%%%%%%%%%%%%%%%%%%%%%
\appendix
\section{Appendix}\label{sec: Appendix}\label{sect: Appendix}
\subsection{Continued proof of \cref{thm: rank-1_prod}} \label{sec:Thm8Proof}
    Recall from the earlier part of the proof that we have concluded that the scalar valued $s \mapsto \text{trace}(\tf^\dagger(-s)\mathcal{F}(s))$ is holomorphic on its domain.
    Then, for any $R > 0$ define the semicircular contour $\Gamma_R$ in the left half-plane ${\C_-}$ as
    \begin{equation*}
        \Gamma_R := \left\{ z \mid z = iw,~w\in[-R,R]\right\} \cup \left\{z\mid z = Re^{i\theta},\theta \in\left[\dfrac{\pi}{2},\dfrac{3\pi}{2}\right]\right\}.
    \end{equation*}
    Since $ s\mapsto \text {trace}(\tf^\dagger(-s)\mathcal{F}(s))$ is scalar-valued and holomorphic on its domain, we can write
    \begin{equation*}
        \dfrac{1}{2\pi}\int_{-\infty}^\infty \text{trace}( \tf^\dagger(-iw)\mathcal{F}(iw)) \, \mathrm{d}w = \lim_{R\to \infty}\dfrac{1}{2\pi i}\int_{\Gamma_R} \text{trace}( \tf^\dagger(-s)\mathcal{F}(s)) \, \mathrm{d}s.
    \end{equation*}
    Note that the singularities of the function $\text{trace}( \tf^\dagger(-\cdot)\mathcal{F}(\cdot))$ are singularities of $\tf^\dagger(-\cdot)$ and $\mathcal{F}(\cdot)$. Moreover because the singularities of $\tf$ were on $\C_-$, the singularities of $\tf^\dagger(-\cdot)$ are on $\C_+$. Since $\Gamma_R \in \C_-$ for any $R>0$, only the singularities of $\mathcal{F}$ will be in the contour. Then, we can evaluate the integral using the residue theorem
    \begin{equation*}
        \lim_{R\to \infty}\dfrac{1}{2\pi i}\int_{\Gamma_R} \text{trace}( \tf^\dagger(-s)\mathcal{F}(s))\, \mathrm{d}s = \text{res}[\text{trace}( \tf^\dagger(-s)\mathcal{F}(s)),\lambda].
    \end{equation*}
    One can simplify $\text{trace}( \tf^\dagger(-s)\mathcal{F}(s))$ to calculate 
    \begin{equation*}
        \text{res}[\text{trace}( \tf^\dagger(-s)\mathcal{F}(s)),\lambda] = \lim_{s\to \lambda} (s-\lambda)\cdot\text{trace}( \tf^*(-s)\mathcal{F}(s)).
    \end{equation*}
    Indeed by the definition of $\mathcal{F}(s)$ we have
    \begin{equation*}
    \begin{aligned}
        \text{trace}( \tf^\dagger(-s)\mathcal{F}(s)) &= \text{trace}(\mathcal{F}(s)\tf^\dagger(-s)) = \text{trace}\left(\dfrac{1}{s-\lambda}\langle \tf^\dagger(-s)[\cdot] ,b\rangle_{\sin}c\right).
    \end{aligned}
    \end{equation*}
    Note that $p\mapsto\dfrac{1}{s-\lambda}\langle \tf^\dagger(-s)p ,b\rangle_{\sin}c$ is a rank-1 operator. Thus, the trace of it can be calculated as
    \begin{equation*}
        \text{trace}\left(\dfrac{1}{s-\lambda}\langle \tf^*(-s)[\cdot] ,b\rangle_{\sin}c\right) = \dfrac{1}{s-\lambda}\langle \tf^\dagger(-s)[c] ,b\rangle_{\sin}.
    \end{equation*}
    This can be seen from the definition of the trace. Note that the definition is independent of the basis functions $p_i$ at \cref{def: trace}. Hence, one can generate a basis $\{p_i\}_{i\in \N}$ that contains $\dfrac{c}{\|c\|}\in \sout$. Then we have
    \begin{equation*}
    \begin{aligned}
        \text{trace}\left(\dfrac{1}{s-\lambda}\langle \tf^*(-s)[\cdot] ,b\rangle_{\sin}c\right) &= \dfrac{1}{s-\lambda}\sum_{i=1}^\infty\langle \tf^\dagger(-s)[p_i] ,b\rangle_{\sin}\langle c, p_i\rangle_{\sout} \\
        &= \dfrac{1}{s-\lambda}\langle \tf^\dagger(-s)\left[\dfrac{c}{\|c\|}\right] ,b\rangle_{\sin}\left\langle c, \dfrac{c}{\|c\|}\right\rangle_{\sout}\\
        &= \dfrac{1}{s-\lambda}\langle \tf^\dagger(-s)[c] ,b\rangle_{\sin}.
    \end{aligned}
    \end{equation*}
    In this case, the inner product becomes
    \begin{equation*}
        \begin{aligned}
            \langle \tf, \mathcal{F}\rangle_{\mathcal{H}_2(\sin,\sout)} &= \text{res}[\text{trace}( \tf^*(-s)\mathcal{F}(s)),\lambda] 
            = \lim_{s\to \lambda} (s-\lambda)\cdot\text{trace}( \tf^\dagger(-s)\mathcal{F}(s))\\
            &=\lim_{s\to \lambda} (s-\lambda)\dfrac{1}{s-\lambda}\langle \tf^\dagger(-s)[c] ,b\rangle_{\sin}
            = \langle \tf^\dagger(-\lambda)[c] ,b\rangle_{\sin}
            \\
            & = \langle \tf(-\overline{\lambda})^\dagger[c] ,b\rangle_{\sin} =\langle c ,\tf(-\overline{\lambda})[b]\rangle_{\sout}.
        \end{aligned}
    \end{equation*}
\subsection{Proof of \cref{thm: H2opt}} \label{sec:Thm9Proof}
    Let $\widetilde{\rtf}$ be another transfer function associated with an $ r$-th order stable system. Then
    \begin{equation*}
    \begin{aligned}
        \|\tf-\rtf\|_{\mathcal{H}_2(\sin,\sout)}^2 &\leq \|\tf-\widetilde{\rtf}\|_{\mathcal{H}_2(\sin,\sout)}^2 = \|\tf-\rtf+\rtf-\widetilde{\rtf}\|_{\mathcal{H}_2(\sin,\sout)}^2 \\
        &\leq \|\tf-\rtf\|_{\mathcal{H}_2(\sin,\sout)}^2 + 2~\Re\langle \tf - \rtf, \rtf - \widetilde{\rtf} \rangle_{\sout} + \|\rtf-\widetilde{\rtf}\|_{\mathcal{H}_2(\sin,\sout)}^2,
    \end{aligned}
    \end{equation*}
    which directly yields that
    \begin{equation}\label{eq: norm_ineq}
        0 \leq 2~\Re\langle \tf - \rtf, \rtf - \widetilde{\rtf} \rangle_{\sout} + \|\rtf-\widetilde{\rtf}\|_{\mathcal{H}_2(\sin,\sout)}^2.
    \end{equation}
    To prove \cref{eq: lint} we will pick a specific $\widetilde{\rtf}$. Indeed, given an arbitrary unit vector $\xi\in \sout$, $\ell = 1,\dots, r$ and $\varepsilon > 0$, define $\vartheta$ as
    \begin{equation*}
        \vartheta = \pi - \arg{\langle \xi, (\tf(-\overline{\lambda_\ell}) - \rtf(-\overline{\lambda_\ell}))[b_\ell]\rangle}
    \end{equation*}
    and using that, define $\widetilde{\rtf}$ such that
    \begin{equation*}
        \widetilde{\rtf}(s) - \rtf(s) = \dfrac{\varepsilon e^{i\vartheta}\langle \cdot, b_\ell\rangle}{s-\lambda_\ell}\xi.
    \end{equation*}
    Then, using \Cref{thm: rank-1_prod}, we can write
    \begin{equation*}
    \begin{aligned}
        \langle \tf - \rtf, \rtf - \widetilde{\rtf} \rangle_{\sout} &= \langle \tf - \rtf, \dfrac{\varepsilon e^{i\vartheta}\langle \cdot, b_\ell\rangle}{s-\lambda_\ell}\xi\rangle_{\sout}
        \\
        &= \varepsilon e^{i\vartheta}\langle (\tf(-\overline{\lambda_\ell}) - \rtf(-\overline{\lambda_\ell}))[b_\ell], \xi\rangle_{\sout} \\
        &=-\varepsilon \dfrac{\langle \xi,(\tf(-\overline{\lambda_\ell}) - \rtf(-\overline{\lambda_\ell}))[b_\ell]\rangle_{\sout}}{|\langle \xi,(\tf(-\overline{\lambda_\ell}) - \rtf(-\overline{\lambda_\ell}))[b_\ell]\rangle_{\sout}|}\langle (\tf(-\overline{\lambda_\ell}) - \rtf(-\overline{\lambda_\ell}))[b_\ell],\xi\rangle_{\sout} \\
        &=-\varepsilon \dfrac{\overline{\langle (\tf(-\overline{\lambda_\ell}) - \rtf(-\overline{\lambda_\ell}))[b_\ell],\xi\rangle_{\sout}}}{|\langle \xi,(\tf(-\overline{\lambda_\ell}) - \rtf(-\overline{\lambda_\ell}))[b_\ell]\rangle_{\sout}|}\langle (\tf(-\overline{\lambda_\ell}) - \rtf(-\overline{\lambda_\ell}))[b_\ell],\xi\rangle_{\sout} \\
        &= -\varepsilon |\langle \xi,(\tf(-\overline{\lambda_\ell}) - \rtf(-\overline{\lambda_\ell}))[b_\ell]\rangle_{\sout}|.
    \end{aligned}
    \end{equation*}
    Moreover, for this $\widetilde{\rtf}$, we can also write $\|\rtf - \widetilde{\rtf}\|_{\mathcal{H}_2(\sin,\sout)}^2$ as 
    \begin{equation*}
        \|\rtf - \widetilde{\rtf}\|_{\mathcal{H}_2(\sin,\sout)}^2 = \varepsilon^2\dfrac{\|b_\ell\|_{\sin}^2}{-2~\Re(\lambda_\ell)}.
    \end{equation*}
    Then by \cref{eq: norm_ineq} we have
    \begin{equation*}
            0\leq |\langle \xi,(\tf(-\overline{\lambda_\ell}) - \rtf(-\overline{\lambda_\ell}))[b_\ell]\rangle_{\sout}| \leq \varepsilon\dfrac{\|b_\ell\|_{\sin}^2}{-4~\Re(\lambda_\ell)}.
    \end{equation*}
    Since this is true for any $\varepsilon$ and $\xi$, we get
    \begin{equation*}
        (\tf(-\overline{\lambda_\ell}) - \rtf(-\overline{\lambda_\ell}))[b_\ell] = 0,
    \end{equation*}
    proving \cref{eq: lint}. For \cref{eq: rint}, we will use the same argument. Indeed, given an arbitrary unit vector $\xi\in \sin$, $\ell = 1,\dots, r$ and $\varepsilon > 0$, define $\vartheta$ as
    \begin{equation*}
        \vartheta = \pi - \arg{\langle c_\ell, (\tf(-\overline{\lambda_\ell}) - \rtf(-\overline{\lambda_\ell}))[\xi]\rangle},
    \end{equation*}
    and then define $\widetilde{\rtf}$ such that
    \begin{equation*}
        \widetilde{\rtf}(s) - \rtf(s) = \dfrac{\varepsilon e^{i\vartheta}\langle \cdot, \xi\rangle}{s-\lambda_\ell}c_\ell.
    \end{equation*}
    Again, using \Cref{thm: rank-1_prod} we can write
    \begin{equation*}
        \langle \tf - \rtf, \rtf - \widetilde{\rtf} \rangle_{\sout} = -\varepsilon |\langle c_\ell,(\tf(-\overline{\lambda_\ell}) - \rtf(-\overline{\lambda_\ell}))[\xi]\rangle_{\sout}|.
    \end{equation*}
    Then by \cref{eq: norm_ineq} we obtain
    \begin{equation*}
        |\langle c_\ell,(\tf(-\overline{\lambda_\ell}) - \rtf(-\overline{\lambda_\ell}))[\xi]\rangle_{\sout}| =0
    \end{equation*}
    for any $\varepsilon>0$ and $\xi \in \sout$. Equivalently, one can write this as
    \begin{equation*}
        |\langle (\tf(-\overline{\lambda_\ell}) - \rtf(-\overline{\lambda_\ell}))^\dagger [c_\ell],\xi\rangle_{\sout}| =0
    \end{equation*}
    for any $\varepsilon>0$ and $\xi \in \sout$. Hence, this leads to
    \begin{equation*}
        (\tf(-\overline{\lambda_\ell}) - \rtf(-\overline{\lambda_\ell}))^*c_\ell = 0
    \end{equation*}
    proving \cref{eq: rint}. For \cref{eq: dint} we will keep the residual vectors fixed and perturb the pole $\lambda_\ell$. For contradiction, assume that \cref{eq: dint} does not hold. Then we can pick $0 < \varepsilon < |\Re(\lambda_\ell)|$ in such a way that $\mu = \lambda_\ell + \varepsilon e^{i\vartheta}$ does not coincide with any of the poles $\lambda_k$ where $\vartheta$ is given by
    \begin{equation*}
        \vartheta = -\arg{\left\langle c_\ell, \left(\dfrac{\mathrm{d}}{\mathrm{d}s}\tf(-\overline{\lambda_\ell}) - \dfrac{\mathrm{d}}{\mathrm{d}s}\rtf(-\overline{\lambda_\ell})\right)[b_\ell]\right\rangle}_{\sout}.
    \end{equation*}
    Note that for small $\varepsilon$, $\Re(\mu) < 0$. Then we define $\widetilde{\rtf}$ such that
    \begin{equation*}
        \rtf(s) - \widetilde{\rtf}(s) = \left(\dfrac{1}{s - \lambda_\ell} -\dfrac{1}{s-\mu}\right)\langle \cdot, b_\ell\rangle_{\sin} c_\ell.
    \end{equation*}
    Then, by \cref{thm: rank-1_prod} we obtain
    \begin{equation*}
        \langle \tf - \rtf, \rtf - \widetilde{\rtf} \rangle_{\sout} = \langle (\tf(-\overline{\lambda_\ell}) - \rtf(-\overline{\lambda_\ell}))[b_\ell], c_\ell\rangle_{\sout} - \langle  (\tf(-\overline{\mu}) - \rtf(-\overline{\mu}))[b_\ell], c_\ell\rangle_{\sout}.
    \end{equation*}
    However, since $\rtf$ is $\mathcal{H}_2$ optimal, it satisfies \cref{eq: lint} (or \cref{eq: rint}). Hence we get
    \begin{equation}\label{eq: ip_dint}
         \langle \tf - \rtf, \rtf - \widetilde{\rtf} \rangle_{\sout} =- \langle  (\tf(-\overline{\mu}) - \rtf(-\overline{\mu}))[b_\ell], c_\ell\rangle_{\sout}.
    \end{equation}
    One can also calculate the norm $\|\rtf - \widetilde{\rtf}\|_{\mathcal{H}_2(\sin,\sout)}^2$ as
    \begin{equation}\label{eq: norm_dint}
        \|\rtf - \widetilde{\rtf}\|_{\mathcal{H}_2(\sin,\sout)}^2 = -\dfrac{|\lambda_\ell - \mu|^2}{|\lambda_\ell + \overline{\mu}|^2}\dfrac{\Re(\mu + \lambda_\ell)}{2~\Re(\mu)\Re(\lambda_\ell)}\|c_\ell\|^2_{\sout}\|b_\ell\|^2_{\sin}.
    \end{equation}
    Plugging \cref{eq: norm_dint,eq: ip_dint} into \cref{eq: norm_ineq} yields
    \begin{equation}\label{eq: dint_ineq}
        0\leq -2~\Re \langle (\tf(-\overline{\mu}) - \rtf(-\overline{\mu}))[b_\ell], c_\ell\rangle_{\sout} -\dfrac{|\lambda_\ell - \mu|^2}{|\lambda_\ell + \overline{\mu}|^2}\dfrac{\Re(\mu + \lambda_\ell)}{2~\Re(\mu)\Re(\lambda_\ell)}\|c_\ell\|^2_{\sout}\|b_\ell\|^2_{\sin}.
    \end{equation}
    Since $\tf$ and $\rtf$ are holomorphic around $-\overline{\lambda_\ell}$ we can expand $\tf$ and $\rtf$ around $-\overline{\lambda_\ell}$. Then for small $\varepsilon>0$ we have
    \begin{align}
        \tf(-\overline{\mu}) &= \tf(-\overline{\lambda_\ell}) + (\overline{\lambda_\ell} - \overline{\mu})\dfrac{\d}{\d s}\tf(-\overline{\lambda_\ell}) +\mathcal{O}(\varepsilon^2),\label{eq: tayl_full} \\
        \rtf(-\overline{\mu}) &= \rtf(-\overline{\lambda_\ell}) + (\overline{\lambda_\ell} - \overline{\mu})\dfrac{\d}{\d s}\rtf(-\overline{\lambda_\ell}) +\mathcal{O}(\varepsilon^2).\label{eq: tayl_red}
    \end{align}
    Hence plugging \cref{eq: tayl_full,eq: tayl_red} to \cref{eq: dint_ineq} yields
    \begin{equation*}
        0 \leq -2\varepsilon  \left|\left\langle \left(\dfrac{\mathrm{d}}{\mathrm{d}s}\tf(-\overline{\lambda_\ell}) - \dfrac{\mathrm{d}}{\mathrm{d}s}\rtf(-\overline{\lambda_\ell})\right)[b_\ell], c_\ell\right\rangle_{\sout}\right| + \mathcal{O}(\varepsilon^2)
    \end{equation*}
    for every $\varepsilon > 0$. Taking the limit as $\varepsilon \to 0$ proves $\cref{eq: dint}$ by contradiction.

\addcontentsline{toc}{section}{References}
\bibliographystyle{plainurl}
\bibliography{refs}

\end{document}